\documentclass[11pt]{amsart}
\usepackage{enumerate,amsmath,amsthm,amsfonts,amssymb,mathrsfs,paralist,lscape}
\usepackage[usenames, dvipsnames]{color}
\usepackage[margin=1in]{geometry} 
\usepackage[colorlinks=true, pdfstartview=FitV, linkcolor=black, citecolor=black, urlcolor=black]{hyperref}

\usepackage{stmaryrd}
\usepackage{tikz}
\usetikzlibrary{matrix,arrows}
\usepackage{blkarray}

\numberwithin{equation}{section}
\newtheorem{theorem}[equation]{Theorem}

\newtheorem{proposition}[equation]{Proposition}
\newtheorem{lemma}[equation]{Lemma}
\newtheorem{corollary}[equation]{Corollary}
\newtheorem{question}[equation]{Question}

\theoremstyle{definition}
\newtheorem{definition}[equation]{Definition}

\newtheorem{example}[equation]{Example}
\newtheorem{remark}[equation]{Remark}
\newtheorem{convention}[equation]{Convention}

\newcommand{\red}[1]{\textcolor{red}{#1}}

\newcommand{\ZZ}{\mathbb{Z}}
\newcommand{\CC}{\mathbb{C}}
\newcommand{\bd}{\mathbf{d}}
\newcommand{\kk}{\Bbbk}
\renewcommand{\tilde}{\widetilde}
\newcommand{\ol}[1]{\overline{#1}}

\makeatletter
\def\Ddots{\mathinner{\mkern1mu\raise\p@
\vbox{\kern7\p@\hbox{.}}\mkern2mu
\raise4\p@\hbox{.}\mkern2mu\raise7\p@\hbox{.}\mkern1mu}}
\makeatother

\DeclareMathOperator{\rank}{rank}
\DeclareMathOperator{\Mat}{Mat}
\DeclareMathOperator{\srep}{srep}

\DeclareMathOperator{\rep}{rep}
\DeclareMathOperator{\im}{im}
\DeclareMathOperator{\Tor}{Tor}
\DeclareMathOperator{\Spec}{Spec}

\title{Equivariant geometry of symmetric quiver orbit closures}
\author{Ryan Kinser}
\address{Ryan Kinser \\ Department of Mathematics\\ University of Iowa \\ Iowa City, IA, USA }
\email{ryan-kinser@uiowa.edu}
\thanks{This work was supported by a grant from the Simons Foundation (636534, RK). This material is based upon work supported by the National Science Foundation under Award No. DMS-2303334.}

\author{Martina Lanini}
\address{Martina Lanini\\ Department of Mathematics\\ University of Rome ``Tor Vergata"\\ Via della Rcerca Scientifica 1, 00133 Rome, Italy}
\email{lanini@mat.uniroma2.it}
\thanks{M.L. acknowledges the 
PRIN2022 CUP E53D23005550006, and the MUR Excellence Department Project 2023--2027 awarded to the Department of Mathematics, University of Rome Tor Vergata, CUP E83C18000100006.}

\author{Jenna Rajchgot}
\address{Jenna Rajchgot\\ Department of Mathematics and Statistics\\ McMaster University\\ Hamilton, ON, Canada}
\email{rajchgoj@mcmaster.ca}
\thanks{J.R. was partially supported by NSERC Discovery Grants 2017-05732 and 2023-04800.}

\subjclass[2020]{16G20, 14M12, 13C40, 14M27}

\keywords{Symmetric quiver, symmetric variety, singularities, degeneration order, Bruhat order, equivariant K-theory}

\begin{document}

\begin{abstract}
We unify problems about the equivariant geometry of symmetric quiver representation varieties, in the finite type setting, with the corresponding problems for symmetric varieties $GL(n)/K$ where $K$ is an orthogonal or symplectic group. In particular, we translate results about singularities of orbit closures; combinatorics of orbit closure containment; and torus equivariant cohomology and K-theory between these classes of varieties. We obtain these results by constructing explicit embeddings with nice properties of homogeneous fiber bundles over type $A$ symmetric quiver representation varieties into symmetric varieties. 
\end{abstract}

\maketitle

\setcounter{tocdepth}{1}
\tableofcontents

\section{Introduction}
\subsection{Context and motivation}
Throughout the paper we work over an algebraically closed field $\kk$ which is not of characteristic 2.

This paper unifies problems about the equivariant geometry of symmetric quiver representation varieties, in Dynkin type $A$, 
with the corresponding problems for the symmetric varieties $GL(n)/K$ where $K=O(n)$ or $Sp(n)$.
The specific problems we consider include:
\begin{enumerate}[(i)]
    \item characterizing singularities of orbit closures;
    \item combinatorial models for the poset of orbit closures with respect to containment;
    \item study of classes of orbit closures in equivariant cohomology and $K$-theory. 
\end{enumerate}

Our work extends a series of investigations connecting ordinary quivers of Dynkin type $A$ with flag varieties $GL(n)/B$ over $\kk$ \cite{MS83,Zel85,LM98,BZ01,BZ02,KMS,KR15,KKR,XZ}, and ordinary quivers of Dynkin type $D$  with symmetric varieties $GL(a+b)/\left(GL(a) \times GL(b)\right)$ over $\kk$ \cite{BZ02,KR21},
to the setting of symmetric quivers.
In the particular case of $\kk=\CC$, the present paper completes the story of connections between symmetric quotients of $GL(n)$ and quiver representation varieties, since these are the only such symmetric quotients.
\begin{center}
\begin{tabular}{|c|c|}
    \hline
    \text{Family of symmetric quotient varieties }&\text{Family of quiver representation varieties}\\
    \hline
    $GL(n)/B$ & Dynkin type $A$\\
    $GL(a+b)/(GL(a)\times GL(b))$ & Dynkin type $D$\\
    $GL(n)/O(n)$ & symmetric Dynkin type $A$ with $\epsilon = 1$\\
    $GL(n)/Sp(n)$ & symmetric Dynkin type $A$ with $\epsilon = -1$\\
    \hline
\end{tabular}
\end{center}
Lusztig has also initiated investigation of a similar connection between varieties of nilpotents representations of cyclically oriented type $\widetilde{A}$ quivers and affine flag varieties
\cite{Lusztig81,Lusztig90,Magyar}.
We note that while there are some results on orbit closures in Dynkin type $E$ quiver representation varieties \cite{Sutar15, Zwara05, Lor21}, there is no known analogue of our work for them yet.

Symmetric quivers and their representations were introduced by Derksen and Weyman in \cite[\S2]{DW02}.
Related notions and applications have appeared in \cite{Zub05,Shmel06,Enomoto09,Bock10,Sam12,Aragona12,Aragona13,BC21,BC21b,LSvdB,RRT,Ciliberti,BCIFF,Reineke}.
Roughly, a symmetric quiver is a quiver equipped with the additional data of an orientation \emph{reversing} involution, and a sign function on vertices and (unoriented) edges fixed by the involution.  
Representations of a symmetric quiver are representations of the underlying quiver which respect this additional data in an appropriate sense; in the setting of this paper, it will amount to restricting certain matrix spaces to (skew-)symmetric matrices, and certain base change groups to orthogonal or symplectic subgroups.
See Section \ref{sec:symquivandrep} for a detailed recollection.


We restrict our attention to symmetric quivers of Dynkin type $A$, 
since they are the only ones whose representation varieties all have finitely many orbits by \cite[Theorem 3.1]{DW02}.
In this setting, the involution is uniquely determined, so the additional data making a type $A$ quiver into a symmetric quiver is just a sign $\epsilon=\pm 1$ (see \S\ref{sec:symreps} for more detail).  Thus we will denote a type $A$ symmetric quiver by $Q^\epsilon$, and for a (symmetric) dimension vector $\bd$, we denote the associated symmetric representation space by $\srep_{Q}^\epsilon(\bd)$ (see \S\ref{sec:symrepvarieties}).

\subsection{Summary of results}
Our main results are summarized in the following theorem, retaining the notation from above.
We typically omit the underlying field $\kk$ from the notation.

\begin{theorem}\label{thm:mainTheorem}
Given a symmetric type $A$ quiver $Q^\epsilon$ with dimension vector $\bd$, there is an associated $G = GL(n)$, a symmetric subgroup 
\[
K_\epsilon\simeq\begin{cases}
O(n) & \text{if }\epsilon =1, \\
Sp(n) & \text{if }\epsilon =-1,
\end{cases}
\]
and parabolic subgroup $P\subseteq G$ such that all the following hold.
\begin{enumerate}[(i)]
    \item There is an injective, order preserving map of partially ordered sets
 \begin{equation}\label{eq:reptoKG}
\begin{split}
\left\{\begin{tabular}{c} $G(\bd)$-orbit closures\\ in $\srep_{Q}^\epsilon(\bd)$\end{tabular} \right\}
& \hookrightarrow
\left\{\begin{tabular}{c} $P$-orbit closures\\ in $G/K_\epsilon$ \end{tabular} \right\}
\end{split}
\end{equation}
which we denote by $\mathscr{O}\mapsto \mathscr{O}^\dagger$ below.
\item There is a smooth affine variety $X$ and a smooth morphism  
\begin{equation}
\varphi: X \times \srep_{Q}^\epsilon(\bd)\rightarrow \srep_{Q}^\epsilon(\bd)^\dagger\subset G/K_\epsilon    
\end{equation}
which restricts to a smooth morphism 
\begin{equation}
        X \times \mathscr{O} \rightarrow \mathscr{O}^\dagger
\end{equation}
for each $\mathscr{O}$ as above. So, any smooth equivalence class of singularity occurring in $\mathscr{O}$ also occurs in $\mathscr{O}^\dagger$. 
\item There is a homomorphism of equivariant Grothendieck groups
\[
K_T(G/K_\epsilon) \to {K}_{T(\bd)}(\srep_{Q}^\epsilon(\bd))
\]
sending the class of the structure sheaf $[\mathcal{O}_{\mathscr{O}^\dagger}]$ to $[\mathcal{O}_\mathscr{O}]$ whenever $\mathscr{O}^\dagger$ is Cohen-Macaulay.
Here, $T$ is the maximal torus of diagonal matrices in $G$ and $T(\bd)$ is the maximal torus of matrices which are diagonal in each factor of $G(\bd)$. 
\item There is a homomorphism of equivariant Chow groups
\[
 A_T(G/K_\epsilon) \to A_{T(\bd)}(\srep_{Q}^\epsilon(\bd))
\]
sending the class $[{\mathscr{O}^\dagger}]$ to $[\mathscr{O}]$ (regardless of whether or not $\mathscr{O}^\dagger$ is Cohen-Macaulay).
\end{enumerate}
\end{theorem}

Theorem \ref{thm:mainTheorem} is proved in Section \ref{sec:mainThmProof}, utilizing additional techniques beyond those in our previous work in the ordinary types $A$ and $D$ setting.

For ordinary representations of Dynkin quivers of types $A$ and $D$,
it is known that orbit closures are always normal and Cohen-Macaulay \cite{Brion01,BZ01,BZ02}. 
The case of type $E$ quivers is open, with some things known about orbit closure singularities \cite{Bon94,Zwara05a,Zwara05,Zwara11,RZ13,Sutar15,LW19,KL22,Lor21}. Not much is known concerning singularities of symmetric quiver orbit closures beyond the cases of $A_2$ and $A_3$. 
Type $A_2$ is classical and well-understood; these varieties are set-theoretically defined by fixed-size minors of a symmetric or skew-symmetric matrix. In type $A_3$ symmetric quiver orbit closures are also certain orthogonal and symplectic analogues of determinantal ideals. These varieties are studied in \cite{LOVETT2005416, MR2309889}, and more recently in \cite{Lorincz}. As noted in the introduction of \cite{Lorincz}, some special cases of $A_3$ symmetric quiver orbit closures have also been of interest in the commutative algebra community as nullcones of certain natural actions  \cite{KS14, HJPS23, PTW23}.

Combining our main theorem with
results of Brion in \cite{Brion01,Brion03}
we obtain the following corollary, which appears as Corollary \ref{cor:sing} in the main body of the paper:
\begin{corollary}
In the case $\epsilon=-1$ of Theorem \ref{thm:mainTheorem}, symmetric quiver orbit closures are normal and Cohen-Macaulay. Under the additional assumption that the characteristic of the field $\kk$ is $0$, symmetric quiver orbit closures have rational singularities.
\end{corollary}

We also obtain a combinatorial model for the poset of orbit closures. 
For ordinary representations of Dynkin quivers of type $A$ of any orientation, a combinatorial model was given in \cite{KR15} using a subposet of the symmetric group $\mathfrak{S}_n$ with respect to Bruhat order.  This model generalized the one of Zelevinsky for equioriented type $A$ quivers \cite{Zel85}.
For Dynkin quivers of type $D$, a combinatorial model was given in \cite{KR21} using a subposet of clans \cite{MO88,Yama97,Wyser16}.
Clan posets can also be encoded as signed involutions, though the partial order is not the Bruhat order in that encoding.

In Section \ref{sec:zperms}, we introduce the symmetric Zelevinsky permutation $v^\epsilon(W)$ associated to a symmetric type $A$ quiver representation $W$ (Definition \ref{def:Zperm}).  Fixing $Q^\epsilon$ and $\bd$, we can think of $v^\epsilon(\cdot)$ as a map from the set of orbit closures in $\srep_{Q}^\epsilon(\bd)$ to the symmetric group $\mathfrak{S}_{2d}$, where $d$ is the total dimension of $\bd$.
We prove the following theorem, retaining the notation of Theorem \ref{thm:mainTheorem}. See Section \ref{sec:zperms} for the proof of this result. 

\begin{theorem}\label{thm:combin}
The map $v^\epsilon(\cdot)$ defines an order-reversing 
injective map from 
the poset of orbit closures in $\srep_{Q}^\epsilon(\bd)$
to the symmetric group $\mathfrak{S}_{2d}$ with Bruhat order.

The image of this map $\mathcal{A}_\epsilon \subset \mathfrak{S}_{2d}$ is explicitly given in \S\ref{sec:Zpermoverview} for $Q$ bipartite, and for general $Q$ the image will be an upper order ideal of the image for the associated bipartite quiver (as in \S\ref{sec:bipartitereduction}).
\end{theorem}

We note that Boos and Cerulli Irelli studied orbit closure containment for symmetric type $A$ quivers in \cite{BC21, BC21b}.  They proved that the degeneration order on symmetric type $A$ quiver representations is induced by the degeneration order on standard type $A$ quiver representations.  In the $\epsilon=1$ case, we use one of their results in our proof, but all of our results in the $\epsilon=-1$ case are independent of their work.  In both cases, we provide a new explicit combinatorial model for the orbit closure poset in terms of Bruhat order on certain sets of involutions.

In \S\ref{sec:converse}, we also prove two theorems which are in some sense ``converse theorems'' to our main result, Theorem \ref{thm:mainTheorem}.
These are more straightforward, but we include them for completeness of the story.
If we interpret our main result as saying that solutions to problems (i), (ii), (iii) above for $G(\bd)$ acting on $\srep_{Q}^\epsilon(\bd)$ can be obtained by specializing solutions to the corresponding problems for an appropriate Borel subgroup $B$ acting on $GL(n)/K_\epsilon$, then we can summarize our converse theorems as follows:
Theorem \ref{thm:converse1} says that solutions to these problems for $K_\epsilon$ acting on $GL(n)/B$ can be obtained by specializing solutions for an appropriate $G(\bd)$ acting on $\srep_{Q}^\epsilon(\bd)$, while
Theorem \ref{thm:converse2} says that solutions to these problems for $B$ acting on spaces of (skew-)symmetric matrices can be obtained by specializing solutions for an appropriate $G(\bd)$ acting on $\srep_{Q}^\epsilon(\bd)$.  For $\epsilon=1$, on the symmetric variety side, work of Pin \cite{Pin01} gives a way of producing $B$-orbit closures which are neither normal nor Cohen-Macaulay, so we can apply the converse theorems to produce non normal and non Cohen-Macaulay symmetric quiver orbit closures. The first example of a non Cohen-Macaulay type $A_3$ symmetric quiver orbit closure appeared in work of Lovett \cite{LOVETT2005416}. In the recent work \cite{Lorincz}, Lorincz characterized when a type $A_3$ symmetric quiver orbit closure is normal and when it is Cohen-Macaulay. While we do not provide a characterization of which symmetric type $A$ quiver orbit closures are normal and which are Cohen-Macaulay, our constructions produces many examples beyond the $A_3$ setting.

\subsection*{Acknowledgements}
This project originated at the meeting \emph{Beyond Toric Geometry} which took place at Casa Mathem\'{a}tica Oaxaca (CMO) in 2017. We thank CMO for their hospitality.

We thank Jacopo Gandini for help with references on symmetric varieties, Zachary Hamaker for teaching us about symmetric and skew-symmetric matrix Schubert varieties, Patricia Klein for help with Lemma \ref{lem:CMthing}, and Jerzy Weyman for making us aware of \cite{LOVETT2005416}.

\section{Symmetric quivers and their representations}
\label{sec:symquivandrep}
Symmetric quivers and their representations were introduced by Derksen and Weyman in \cite{DW02}. In this section, we recall these concepts, mostly following \cite[\S2.2]{Sam12}.
All vector spaces, linear maps, and so forth are over $\kk$.

\subsection{Symmetric quivers}\label{sec:symquiv}
Fix a quiver $Q$ with vertex set $Q_0$ and arrow set $Q_1$. 
We write $ta$ and $ha$ for the \emph{tail} and \emph{head} of an arrow $ta \xrightarrow{a} ha$. 
An introduction to the theory of quiver representations can be found in the textbooks \cite{ASS06,Schiffler14,DW17}.

Let $\tau\colon Q \to Q$ be an orientation reversing involution. Formally, we have $\tau\colon Q_0 \to Q_0$ and $\tau\colon Q_1 \to Q_1$ satisfying $\tau(ha) = t \tau (a)$ and $\tau(ta) = h \tau (a)$ for all $a \in Q_1$.  
We denote the fixed vertices and arrows by $Q_0^\tau$ and $Q_1^\tau$ as usual, and choose a sign function $s\colon Q_0^\tau \cup Q_1^\tau \to \{\pm 1\}$.  
The triple $(Q, \tau, s)$ is called a \emph{symmetric quiver}, which we denote by just $Q$ since $\tau$ and $s$ are typically fixed.
We also choose a partition of each of the vertex and arrow sets
\begin{equation}\label{eq:partition}
Q_i = Q_i^+ \amalg Q_i^\tau \amalg Q_i^-, \qquad i=0,1
\end{equation}
which satisfies $\tau(Q_i^{\pm}) = Q_i^{\mp}$, and further use the sign function to partition the fixed point sets as
\begin{equation}
Q_i^{\tau \pm} = \{x \in Q_i^\tau \mid s(x)=\pm 1\}.
\end{equation}

\begin{convention}\label{con:A}
The quivers of interest in this paper are of Dynkin type $A$, thus their underlying graph each possesses a unique nontrivial involution.  When we reference a symmetric type $A$ quiver, it will always be understood that $\tau$ is this one without explicitly describing $\tau$.

Furthermore, we always assume that the partitions of vertices and arrows in \eqref{eq:partition} are chosen so \emph{positive arrows} (those in $Q_1^+$) have only positive, or fixed, vertices as their heads and tails.
This implies the similar statement where ``positive'' is replaced with ``negative'' is also true.
One may visualize the type $A$ quiver as being horizontally embedded on the page with the ``positive half'' on the left and ``negative half'' on the right. For example, if
\begin{equation}\label{eqn:expleQuiverA4}
Q =\vcenter{\hbox{ \begin{tikzpicture}[point/.style={shape=circle,fill=black,scale=.5pt,outer sep=3pt},>=latex]
   \node[outer sep=-2pt] (y0) at (-1,0) {$1$};
   \node[outer sep=-2pt] (x1) at (0,1) {$2$};
  \node[outer sep=-2pt] (y1) at (1,0) {$3$};
   \node[outer sep=-2pt] (x2) at (2,1) {$4$};
  \path[->]
  	(x1) edge (y0) 
	(x1) edge (y1)
  	(x2) edge (y1);
   \end{tikzpicture}}}
  \end{equation}
then we take $Q_0^+=\{1,2\}$, $Q_1^+=\{2\to 1\}$, $Q_0^-=\{3,4\}$, $Q_1^-=\{4\to 3\}$, and we have $Q_1^\tau=\{2\to 3\}$.
\end{convention}

With this convention, the only additional data that a Dynkin type $A$ symmetric quiver has beyond its underlying quiver is a sign $\epsilon \in \{\pm 1\}$, corresponding to the value of $s$ on the unique fixed arrow or fixed vertex of $\tau$. Thus we can denote a Dynkin type $A$ symmetric quiver by simply $Q^\epsilon$, even omitting the $\epsilon$ at times when its value is fixed.

\subsection{Symmetric quiver representations}\label{sec:symreps}
Let $Q=(Q,\tau,s)$ be a symmetric quiver.
A \emph{dimension vector} $\bd\colon Q_0 \to \mathbb{Z}_{\geq 0}$ is \emph{symmetric} if $\bd(z) = \bd(\tau (z))$ for all $z\in Q_0$.
A \emph{symmetric representation} $V$ of symmetric dimension vector $\bd$ is an assignment of a vector space $V_z$ of dimension $\bd(z)$ to each $z \in Q_0^+ \cup Q_0^\tau$.
Then the dual space $V_z^*$ is assigned to $\tau(z)$.
For this to make sense when $\tau(z)=z$, we fix a nondegenerate bilinear form on $V_z$, which gives an identification $J_z\colon V_z\xrightarrow{\sim} V_z^*$.
This bilinear form is required to be symmetric if $s(z)=+1$ and skew-symmetric if $s(z)=-1$.
Note that this makes $J_z^{-1} = J_z^* = \pm J_z$ where $s(z)=\pm 1$. 
For each $a \in Q_1$, we assign a linear map $V_a \colon V_{ta} \to V_{ha}$ satisfying:
\begin{enumerate}[(i)]
    \item $V_a = s(a)V_a^*$ for $a \in Q_1^\tau$ with $V_{ta} \cong V_{ta}^{**}$ the canonical isomorphism;
    \item $V_a=V_{\tau(a)}^*$ when both $ta, ha \notin Q_0^\tau$;
    \item for $a \in Q_1^+$ with $ha \in Q_0^\tau$, we have $V_{\tau(a)}=V_a^*J_{ha}$, corresponding to the identifications
    \begin{equation}
    V_{t\tau(a)}=V_{\tau(ha)}  =V_{ha} \xrightarrow{J_{ha}} V_{ha}^* \xrightarrow{V_a^*} V_{ta}^* = V_{\tau(ta)}=V_{h\tau(a)};
    \end{equation}
    \item for $a \in Q_1^+$ with $ta \in Q_0^\tau$, we have $V_{\tau(a)}=J_{ta}V_a^*$, corresponding to the identifications
    \begin{equation}
    V_{t\tau(a)}=V_{\tau(ha)}=V_{ha}^* \xrightarrow{V_a^*} V_{ta}^* \xrightarrow{J_{ta}} V_{ta} = V_{\tau(ta)}=V_{h\tau(a)}.
    \end{equation}
\end{enumerate}

\subsection{Symmetric representation varieties}\label{sec:symrepvarieties}
For an ordinary quiver $Q$ and dimension vector $\bd$ we have the \emph{representation space}
\begin{equation}
\rep_Q(\bd) := \prod_{a \in Q_1} \Mat(\bd(ha), \bd(ta)),
\end{equation}
where $\Mat(m, n)$ denotes the algebraic variety of matrices with $m$ rows, $n$ columns, and entries in $\kk$.  
Each $V=(V_a)_{a \in Q_1}$ in $\rep_Q(\bd)$ is a \emph{representation} of $Q$ where we have assigned the vector space $\kk^{\bd(z)}$ at each vertex (in particular with a standard basis), and each matrix $V_a$ maps column vectors in $\kk^{\bd(ta)}$ to column vectors in $\kk^{\bd(ha)}$ by left multiplication.  
 
There is a \emph{base change group} $GL(\bd) := \prod_{z \in Q_0} GL(\bd(z))$,
which acts on $\rep_Q(\bd)$. 
Here $GL(\bd(z))$ denotes the general linear group of invertible $\bd(z)\times \bd(z)$ matrices with entries in $\kk$.
Explicitly, if $g = (g_z)_{z\in Q_0}$ is an element of $GL(\bd)$, and $V = (V_a)_{a\in Q_1} \in \rep_Q(\bd)$, then the (left) action of $GL(\bd)$ on $\rep_Q(\bd)$ is given by
\begin{equation}\label{eq:gdotv}
g\cdot V = (g_{ha}V_a g_{ta}^{-1})_{a\in Q_1}.
\end{equation}
Two points $V, W \in \rep_Q(\bd)$ lie in the same $GL(\bd)$-orbit if and only if $V$ and $W$ are isomorphic as representations of $Q$.  

For a symmetric quiver $Q$ and a symmetric dimension vector $\bd$, we get the \emph{symmetric representation space} $\srep_Q(\bd)$ as follows.
Let $\Mat_{+}(n,n)$ be the space of $n\times n$ symmetric matrices and $\Mat_{-}(n,n)$ the space of $n\times n$ skew-symmetric matrices.  
For convenience in the notation we sometimes interpret subscripts of $1$ and $-1$ as the symbols $+$ and $-$, respectively.
For $z \in Q_0^\tau$, let $\Omega_z \in GL(\bd(z))$ be symmetric if $s(z)=1$ and skew-symmetric if $s(z)=-1$.
In both cases, define the subgroup
\begin{equation}\label{eq:Kz}
    K_z:=\{g \in GL(\bd(z))\mid g^t\Omega_z g = \Omega_z\},
\end{equation}
so $K_z$ is an orthogonal group if $s(z)=1$ and a symplectic group if $s(z)=-1$.

We can then define
\begin{equation}\label{eq:srep}
    \srep_Q(\bd) := \prod_{a \in Q_1^+} \Mat(\bd(ha), \bd(ta))\times 
    \prod_{a \in Q_1^{\tau}} \Mat_{s(a)}(\bd(ha), \bd(ta))
\end{equation}
with corresponding \emph{base change group}
\begin{equation}\label{eq:G}
    G(\bd) := \prod_{z \in Q_0^+} GL(\bd(z)) \times 
    \prod_{z \in Q_0^{\tau}} K_z.
\end{equation}
The action of $G(\bd)$ on $\srep_Q(\bd)$ is given by the same formula as \eqref{eq:gdotv}, except that the index $a$ ranges over $Q^+_1 \cup Q^\tau_1$, and $g_z$ should be interpreted as $g_{\tau(z)}^{-t}$ whenever $z \in Q_0^-$.
Following Convention \ref{con:A}, when $Q$ is type $A$ we sometimes write $\srep_Q^\epsilon(\bd)$.

While symmetric representations are defined basis-free in \S\ref{sec:symreps}, working with representation spaces involves fixing a basis at each vertex.  We now discuss choice of basis and the resulting relations between representation spaces and their symmetric counterparts.
We consider $G(\bd)$ as a closed subgroup of $GL(\bd)$ via the embedding
\begin{equation}\label{eq:GdGLd}
(g_z)_{z \in Q_0^+ \cup Q_0^\tau} \mapsto (h_w)_{w \in Q_0}, \qquad \text{where}\quad
    h_w := 
\begin{cases}
    g_w & \text{if }w  \in Q_0^+ \\
    g_w& \text{if }w  \in Q_0^\tau\\
    g_{\tau(w)}^{-t} & \text{if }w \in Q_0^-. 
\end{cases}
\end{equation}

Moreover, we consider the closed embedding $\srep_Q(\bd) \hookrightarrow \rep_Q(\bd)$
given by
\begin{equation}\label{eq:sreptorep}
(V_a)_{a \in Q_1^+ \cup Q_1^\tau} \mapsto (W_b)_{b \in Q_1}, \qquad \text{where}\quad
    W_b := 
\begin{cases}
    V_b & \text{if }b \in Q_1^+\\
    V_b= s(b)V_b^t & \text{if }b \in Q_1^\tau\\
    \Omega_{t\tau(b)}^{-1} V_{\tau(b)}^t \Omega_{h\tau(b)} & \text{if }b \in Q_1^-
\end{cases}
\end{equation}
where $\Omega_w$ is the matrix representing the fixed nondegenerate bilinear form on $V_w=\kk^{\bd(w)}$ and, for the purposes of this definition, we extend the notation $\Omega_z$ to all $z \in Q_0$ by taking $\Omega_z=1$ for $z \notin Q_0^\tau$.

\begin{example}\label{ex:runExStart}
Consider the quiver from \eqref{eqn:expleQuiverA4}, that is
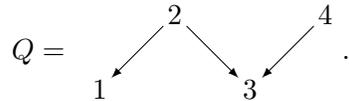
\begin{equation*}
Q =\vcenter{\hbox{ \begin{tikzpicture}[point/.style={shape=circle,fill=black,scale=.5pt,outer sep=3pt},>=latex]
   \node[outer sep=-2pt] (y0) at (-1,0) {$1$};
   \node[outer sep=-2pt] (x1) at (0,1) {$2$};
  \node[outer sep=-2pt] (y1) at (1,0) {$3$};
   \node[outer sep=-2pt] (x2) at (2,1) {$4$};
 
  \path[->]
  	(x1) edge (y0) 
	(x1) edge (y1)
  	(x2) edge (y1);
   \end{tikzpicture}}}.
  \end{equation*}
If $\tau:Q\to Q$ is the unique non trivial involution of $Q$,
we have $Q_0^\tau=\emptyset$, $Q_1^\tau=\{2\to 3\}$. Let us choose the sign $s(2\to 3)=-1$, so that the triple $(Q,\tau, s)$ is a symmetric quiver. 

We fix the symmetric dimension vector $\bd=(1,3,3,1)$. In this case, $\srep_Q(\bd)=\Mat(1,3)\times\Mat_-(3,3)$ and an example of a point in this representation space is 
\[
(A,B) = \left(\begin{bmatrix}1&0&0\end{bmatrix}, \begin{bmatrix}
   0&1&0\\ -1&0&0\\ 0&0&0
   \end{bmatrix}\right).
   \]
The image of this point under the embedding \eqref{eq:sreptorep} is
\[
(A,B,A^T)=\left(\begin{bmatrix}1&0&0\end{bmatrix}, \begin{bmatrix}
   0&1&0\\ -1&0&0\\ 0&0&0
   \end{bmatrix}, \begin{bmatrix}
   1\\0\\ 0
   \end{bmatrix}\right)
\]
which we graphically represent by
\[
W =\vcenter{\hbox{ \begin{tikzpicture}[point/.style={shape=circle,fill=black,scale=.5pt,outer sep=3pt},>=latex]
   \node[outer sep=-2pt] (y0) at (-1,0) {$\kk$};
   \node[outer sep=-2pt] (x1) at (0,1) {$\kk^3$};
  \node[outer sep=-2pt] (y1) at (1,0) {$\kk^3$};
   \node[outer sep=-2pt] (x2) at (2,1) {$\kk$};
 
  \path[->]
  	(x1) edge node[left] {$A$} (y0) 
	(x1) edge node[left] {$B$} (y1)
  	(x2) edge node[right] {$A^T$} (y1);
   \end{tikzpicture}}}.
\]
\end{example}

The following lemma can be verified by direct computation using the definitions above.

\begin{lemma}
The closed embedding $\srep_Q(\bd) \hookrightarrow \rep_Q(\bd)$ described in \eqref{eq:sreptorep} makes $\srep_Q(\bd)$ a $G(\bd)$-invariant closed subvariety of $\rep_Q(\bd)$, with respect to \eqref{eq:GdGLd}.
\end{lemma}

By definition, two points $V, W \in \srep_Q(\bd)$ lie in the same $G(\bd)$-orbit if and only if $V$ and $W$ are isomorphic as symmetric representations of $Q$.  The following theorem of Derksen and Weyman relates this to the ordinary quiver representation setting.

\begin{theorem}\cite[Thm.~2.6]{DW02}\label{thm:DWorbits}
Two points $V,W \in \srep_Q(\bd)$ lie in the same $G(\bd)$-orbit if and only if they lie in the same $GL(\bd)$-orbit.
\end{theorem}

\section{Symmetric varieties, Borel orbit closures, and Mars-Springer slices}\label{sec:symvarieties}
In this section we revise some results on symmetric varieties, focusing on symmetric homogeneous spaces for the action of $GL(2d)$.  Our main references are \cite{RS90, MS98, BH00}.
\subsection{Conventions}\label{subsec:Conventions}
Let $G=GL(2d)$ and let $\Omega \in G$ be symmetric or skew-symmetric.
Consider the involution
\begin{equation}\label{eq:theta}
\theta\colon G \to G, \qquad \theta(g) = \Omega^{-1} g^{-t} \Omega
\end{equation}
and let $K=G^{\theta}$. Therefore, if $\Omega$ is symmetric then $K$ is isomorphic to the group  $O(2d)$ of orthogonal matrices, while if $\Omega$ is skew-symmetric, then  $K$ is isomorphic to the symplectic group $Sp(2d)$ .

Observe that if $\kk=\mathbb{C}$ and $\theta$ is any involutive automorphism of $GL(2d)$, then by \cite[Chapter X, \S6]{Hel78} 
the fixed point group $G^{\theta}$ is isomorphic to one of the following: $O(2d)$, $Sp(2d)$, $GL(p)\times GL(q)$ with $p+q=2d$. The relation between symmetric spaces of the form $GL(m)/\left(GL(p)\times GL(q)\right)$ (with $p+q=m$, $m$ not necessarily even) and quiver representation varieties has been dealt with in \cite{KR21}. The remaining two cases are investigated in the present paper, that therefore completes the task of relating complex symmetric quotients of $GL(2d)$ to quiver representation theory.

\subsection{Mars-Springer slices}\label{subsec:MSslices} We recall here \cite[\S6]{MS98} in the setting $G=GL(2d)$. Let $\Omega, \theta$, and $K$ be as in \S\ref{subsec:Conventions}. From now on, we assume that $\theta$ stabilizes the Borel subgroup $B_+$ of upper triangular matrices and its maximal torus $T$ of diagonal matrices. 
It follows that the opposite Borel $B_-$ of lower triangular matrices and also the normalizer $N$ of $T$ in $G$ are stabilized by $\theta$, so that $\theta$ induces an automorphism of the Weyl group $\mathcal{W}=N/T$, which we denote by the same symbol. Moreover, $\theta$ induces an automorphism of the root system which stabilizes the set of simple roots.

We will be interested in  $B_-$-orbits on the symmetric variety $G/K$, and introduce the set
\begin{equation}
\mathcal{V}=\{x\in G\mid x\theta(x)^{-1}\in N\}.
\end{equation}
The left and right actions of $G$ on itself restrict to a left $T$-action and a right $K$-action on $\mathcal{V}$. We obtain a set bijection
\begin{equation}
T\setminus\mathcal{V}/K \rightarrow  B_-\!\setminus G/K, \qquad TxK\mapsto B_-xK.
\end{equation}
For any $x\in \mathcal{V}$, we define another involutive automorphism $\psi_x$ of $G$. Set $\eta_x:=x\theta(x)^{-1}$ and define
\begin{equation}\label{eqn:defPsi}
  \psi_x: G\rightarrow G, \qquad g\mapsto x\theta(x)^{-1}\theta(g)\theta(x)x^{-1},
\end{equation}
that is, $\psi_x$ is the composition of maps $\theta$ followed by left conjugation by $\eta_x$.
Let $\overline{x}=xK/K$ and let $U$ be the unipotent radical of $B_+$. The corresponding \emph{Mars-Springer slice} is
\begin{equation}\label{eq:Sxbar}
S_{\overline{x}}=(U\cap \psi_x(U))\overline{x}.
\end{equation}

\subsection{The Mars-Springer Lemma}
 We retain the conventions of \S\ref{subsec:Conventions}, and recall in this case the notion of an attractive slice to a $B_-$-orbit.
Let $Z$ be any irreducible variety acted upon by $B_-$ and let $z\in Z$. Following \cite[Definition 2.4]{WWY}, we define an attractive slice to the orbit  $B_-\cdot z$ at $z$ in $Z$ to be a locally closed subvariety $S$ of $Z$ having the following properties:
\begin{enumerate}[(i)]
    \item the restriction of the action map 
   $B_-\times S\rightarrow Z$ is a smooth morphism;
    \item $\dim(S)+\dim(B_-\cdot z)=\dim(Z)$;
    \item there is a one-parameter subgroup 
$\lambda:\mathbb{G}_m\rightarrow B_-$
    such that:
    \begin{enumerate}
        \item $Im(\lambda)$ stabilizes $S$ and fixes the point $z$,
        \item the action map $\mathbb{G}_m\times S\rightarrow S$ can be extended to a morphism of algebraic varieties $\mathbb{A}^1\times S\rightarrow S$ by sending $\{0\}\times S$ to $z$. 
    \end{enumerate}
\end{enumerate}

With this terminology, we recall the following key result of Mars and Springer.

\begin{lemma}(\cite[\S6.4]{MS98})
For $x\in\mathcal{V}$,
the variety $S_{\overline{x}}$ of \eqref{eq:Sxbar} is an attractive slice to 
$B_-xK/K$ 
at $\overline{x}$ in $G/K$.
\end{lemma}
In this paper, we will deal with intersections of certain attractive slices with
$B_-$-orbit closures. The following result is
a consequence of the previous lemma
(see \cite[Lemma 2.6]{WWY} for a complete proof).
\begin{lemma}\label{lem:MSLemma} Let $x\in\mathcal{V}$ and let $Z$ be a 
 $B_-$-orbit closure in $G/K$ such that $\overline{x}\in Z$. Then the scheme theoretic intersection $Z\cap S_{\overline{x}}$ is reduced and $Z\cap S_{\overline{x}}$ is an attractive slice to
 $B_-xK$ at $\overline{x}$ in $Z$.
\end{lemma}

Following \cite{WWY}, we call an intersection $Z\cap S_{\overline{x}}$, as in Lemma \ref{lem:MSLemma}, a   \emph{Mars-Springer variety}. Note that if $P$ is a parabolic subgroup of $G$ containing $B_-$, then any $P$-orbit closure is also a  $B_-$-orbit closure and the previous lemma applies.

\subsection{Symmetric quotients as matrix spaces}\label{sec:embedding}
Recall that we denote by $\Mat_\epsilon(2d,2d)$ the set of symmetric or skew-symmetric $2d\times 2d$-matrices.
We describe here an open immersion of $G/K$ into $\Mat_\epsilon(2d,2d)$.
Consider the automorphism 
\begin{equation}
\sigma:G\rightarrow G, \quad g\mapsto g\theta (g^{-1})=g\Omega^{-1} g^t\Omega.
\end{equation}
As explained in   \cite[\S1.2]{RS90} or \cite[\S6.3]{MS98}, the morphism $\sigma:G\rightarrow \sigma(G)$ is separable and it is clear that $\ker(\sigma)=G^\theta=K$. Thus, we get an isomorphism $G/K\rightarrow \sigma(G)$, and composing with right multiplication by $\Omega^{-1}$ we get an open immersion
\begin{equation}\label{eqn:MPembedding}
\iota_\Omega:G/K\rightarrow \Mat_\epsilon(2d,2d), \quad gK\mapsto g\Omega^{-1} g^t
\end{equation}
whose image is the open subvariety of invertible elements in $\Mat_\epsilon(2d,2d)$, which we denote by $\Mat_\epsilon^\circ(2d,2d)$. 
Note that while the domain of $\iota_\Omega$ depends on $\Omega$, the image in $\Mat_\epsilon(2d,2d)$ does not.
The open immersion $\iota_\Omega$ is $G$-equivariant where the $G$-action on $G/K$ is given by left multiplication and the $G$-action on $\text{Mat}_\epsilon(2d,2d)$ is given by $g\cdot M = gMg^t$, for $g\in G$ and $M\in \text{Mat}_\epsilon(2d,2d)$.

\subsection{A specific Mars-Springer slice}\label{sec:OmegaOmega'}
We now specialize the above generalities to the case needed for our explicit construction of the symmetric Zelevinsky map (see \S\ref{ssec:symmetricZmap}). Let
\begin{equation}\label{eq:omegapm}
\Omega_\epsilon = \begin{bmatrix}0& J\\ \epsilon J&0\end{bmatrix}
\end{equation}
where $J$ is a $d\times d$ matrix with $1$s along the antidiagonal and $0$s elsewhere
 and $\epsilon\in\{\pm 1\}$ as usual.
We denote by $\theta_\epsilon$
 the involutive automorphism of $G$ corresponding to $\Omega_\epsilon$ 
 via \eqref{eq:theta}.
 Moreover, we write $\iota_\epsilon$ for the embedding $\iota_{\Omega_\epsilon}$ from \eqref{eqn:MPembedding}. 

 It is straightforward to check that $\theta_\epsilon$ preserves the Borel subgroup $B_+$ of upper triangular matrices, its opposite $B_-$ and the maximal torus $T$ of diagonal matrices, so that the theory of the preceding subsections can be applied in this case.

Let $x=\begin{bmatrix}
0& J\\ I_d&0
\end{bmatrix}$, where $I_d$ is the $d\times d$ identity matrix.
Computing $\eta_x= x\theta_\epsilon(x)^{-1}=\epsilon  \begin{bmatrix}J&0\\0&J\end{bmatrix}\in N$,
we conclude that $x\in\mathcal{V}$. Let $\psi_x$ be as in \eqref{eqn:defPsi}. 
Note that $\psi_x(U) = x\theta(x)^{-1}U\theta(x)x^{-1}$ since $U$ is $\theta_\epsilon$-stable. 
With a direct computation we obtain
\begin{equation}\label{eq:psiU}
\psi_x(U) =\left\{ \begin{bmatrix}L_1 &M\\0 &L_2 \end{bmatrix} \mid L_1, L_2 \text{ unipotent lower triangular and $M$ arbitrary}\right\}
\end{equation}
and thus we have
\begin{equation}
U\cap \psi_x(U) = \left\{\begin{bmatrix}I_d&M\\0&I_d\end{bmatrix}\mid \text{$M$ arbitrary}\right\}.
\end{equation}

Then the Mars-Springer slice in this case is
\begin{equation}\label{eq:ourMSslice}
S_{\overline{x}} = (U\cap \psi_x(U))\cdot \overline{x} = \left\{\begin{bmatrix}M&\epsilon J\\I_d&0\end{bmatrix}K/K\right\},
\end{equation}
where $M$ is an arbitrary $d\times d$ matrix. 

The morphism $\iota_\epsilon$ is an isomorphism of $G$-varieties if we restrict its codomain to $\text{Mat}_\epsilon^\circ(2d,2d)$, and we denote its inverse by
\begin{equation}\label{eqn: IsoMatrixSpaceSymmSpace}
        j_\epsilon: \text{Mat}_\epsilon^\circ(2d,2d)\rightarrow G/K.
    \end{equation}

\begin{lemma}\label{lem:MSSIsomorphism}
The isomorphism $\iota_\epsilon$ identifies the slice $S_{\overline{x}} \subset G/K$ with $Y_d^\epsilon \subset \Mat_\epsilon^\circ(2d,2d)$, where
\begin{equation}\label{eqn:MatrixSlice}
Y_d^\epsilon :=\left\{
\begin{bmatrix}
\Sigma&I_d\\ \epsilon I_d&0\end{bmatrix}\mid \Sigma \in \Mat_\epsilon(d,d)\right\}.
\end{equation}
Moreover, if $O$ is a $B_-$-orbit closure on $\Mat_\epsilon^\circ (2d,2d)$ that has non-trivial intersection with $Y_d$, then $O\cap Y_d^\epsilon$ is reduced and  there is a smooth map $B_-\times (Y_d^\epsilon\cap O)\rightarrow O$ given by the restriction of the $B_-$-action.  
\end{lemma}
\begin{proof}
We directly check that for any $M\in\Mat(d,d)$ we have
\begin{equation}
\iota_\epsilon\left(\begin{bmatrix}
M&\epsilon J\\ I_d&0
\end{bmatrix}K\right)=\begin{bmatrix}
M^t\epsilon M& I_d\\\epsilon I_d&0.
\end{bmatrix}
\end{equation}
and we conclude that the image of  $S_{\overline{x}}$ under $\iota_\epsilon$ is precisely $Y^\epsilon_d$.

Let us denote  by $Z$ the isomorphic image $j_\epsilon(O)$ of our $B_-$-orbit closure.
 Since the isomorphism $j_\epsilon$ is $B_-$-equivariant, the image $Z$ is a $B_-$-orbit closure in $G/K$. Furthermore, we have previously shown that $j_\epsilon(Y^\epsilon_d)$ is equal to the Mars-Springer slice $S_{\overline{x}}$. Thus, by Lemma \ref{lem:MSLemma}, $Z\cap S_{\underline{x}}=j_\epsilon(O)\cap j_\epsilon(Y^\epsilon_d)$ is reduced. Hence, the isomorphic scheme $O\cap Y^\epsilon_d$ is also reduced.
Finally, by Lemma \ref{lem:MSLemma} (and item (i) of the definition of Mars-Springer slice) we get  a smooth map
\begin{equation}
B_-\times (Z\cap S_{\overline{x}})\rightarrow Z,\quad (b,z)\mapsto b\cdot z.
\end{equation}
Since composition by an isomorphism clearly preserves smooth maps, once again from the $B_-$-equivarance of $j_\epsilon$ we obtain a smooth map 
\begin{equation}
B_-\times (O\cap Y_d^\epsilon)\rightarrow O, \quad (b,o)\mapsto \iota_\epsilon(b\cdot j_\epsilon(o))=\iota_\epsilon\cdot j_\epsilon(b\cdot o)=b\cdot o. \qedhere
\end{equation}
\end{proof}

\section{Borel orbit closures in symmetric varieties via rank conditions}
In this section, we fix our conventions for permutations and associated permutation matrices (Section \ref{subsec:permConventions}). We then recall that Borel orbits and their closures in $G/K$ are determined by certain northwest rank conditions (Section \ref{subsec:BorbitsAndRanks}).

\subsection{Parametrization of Borel orbit closures via involutions}
\label{sect:combBorel}
The poset structure of Borel orbit closures in symmetric varieties was studied in \cite{RS90} in general. 
In this section, we recall a result of Richardson-Springer describing the poset structure of
 $B_-$-orbit closures in $G/K$, in the special case of $G=GL(2d)$ and $K=Sp(2d)$ or $O(2d)$. More precisely, there is an order reversing bijection between the set of $B_-$-orbit closures in $G/K$, partially ordered by inclusion, and a distinguished subset of the symmetric group equipped with the Bruhat order.
 Since in the symplectic case our chosen non degenerate skew symmetric form differs from Richardson-Springer's one, we show that their result still holds with our choice. 

Recall that in \S\ref{subsec:MSslices} we have identified the set of $B_-$-orbits on $G/K$ with the $(T,K)$-double cosets in the set $\mathcal{V}$. 
Since $\Omega_{\epsilon}^{-1}\in N$,  there is a map
\begin{equation}
\mathcal{V}\rightarrow  \mathcal{W}=N/T, \quad y\mapsto (y\theta(y)^{-1}\Omega_\epsilon^{-1}) T=(y\Omega_\epsilon^{-1} y^t)T
\end{equation}
which is constant along $(T,K)$-cosets, so that we get an induced map
\begin{equation}\label{eqn:phi}
\varphi:B_-\setminus G/ K\rightarrow  \mathcal{W}=N/T, \quad B_-yK\mapsto (y\Omega_\epsilon^{-1} y^t)T
\end{equation}
where the set on the left hand side has a poset structure given by the  $B_-$-orbit closure inclusion relation on $G/K$. It is clear that if $y\in\mathcal{V}$, then $y\Omega_\epsilon^{-1} y^t$ is a symmetric/skew-symmetric matrix which belongs to $N$, so that if $w\in \textrm{Im}(\varphi)$, then $w=w^{-1}$.

In order to combinatorially control the orbit closure relation, we identify $\mathcal{W}$ with the permutation group $\mathfrak{S}_{2d}$
and set
\begin{equation}
I^\epsilon:=\left\{
\begin{array}{ll}
 \{ w\in \mathcal{W}\mid w=w^{-1}\}&\hbox{ if } \epsilon=1, \\
  \{ w\in \mathcal{W}\mid w=w^{-1}\hbox{ and }w(i)\neq i \hbox{ for all }i\in [2d]\}   & \hbox{ if } \epsilon=-1.
\end{array}\right.
\end{equation}
For $\Omega=\Omega_+$, it is showed in \cite[Example 10.2]{RS90} that 
that $\varphi$ induces an order reversing bijection  between the set of
$B_-$-orbits on $G/K$, partially ordered by inclusion, and $I^{+}$ with the standard Bruhat order.

In the skew-symmetric case, Richardson and Springer considered instead the matrix
\begin{equation}\label{eq:omegaRS}
\Omega' = \begin{bmatrix}0& I_d\\ -I_d&0\end{bmatrix}.
\end{equation}
Denote by $\theta'$ the corresponding involution of $GL(2d)$.
The following lemma relates $\theta_-$ and $\theta'$ and allows us to apply Richardson-Springer's result to conclude that there is an order reversing bijection between $I^{-}$, endowed with the standard Bruhat order, and the poset of 
 $B_-$-orbits on $G/G^{\theta_-}$. 

\begin{lemma}\label{lem:RStheta}
Let $x= \begin{bmatrix}0& J\\ I_d&0\end{bmatrix}$. Then, for any $g\in G$ we have $\theta_-(g)=x^{-1}\theta'(xgx^{-1})x$.
\end{lemma}
\begin{proof}Let $\Omega''=\begin{bmatrix}
J&0\\0&J
\end{bmatrix}=\Omega''^{-1}$, so that $\Omega'=\Omega''\Omega_-$. Note that
\[
\theta_-(\Omega'')=\Omega'', \quad \Omega''=x^2, \quad\hbox{and }\quad \theta_-(x)=- x^{-1},
\]
hence $\theta_-(\Omega''x)=- x$. Then for $g\in G$ we have
\[
   \theta'(xgx^{-1})=\Omega_{-}'^{-1}(xgx^{-1})^{-t}{\Omega'}=\Omega_-^{-1}\Omega''(xg^{-t}x^{-1}){\Omega''}^{-1}\Omega_-=\theta_-(\Omega''x)\theta_-(g)\theta_-(\Omega''x)^{-1}=x\theta_-(g)x^{-1}. \qedhere
\]
\end{proof}

Let $B'_-:=xB_-x^{-1}$, $K'=G^{\theta'}$ and 
\[
\varphi':B'_-\setminus G/K'\rightarrow I^\epsilon, \quad B'_-yK'\mapsto (y\Omega'^{-1}y^t)T.
\]
By \cite[Propostion 10.4.1]{RS90}, $\varphi'$ is an order reversing bijection. Therefore, by Lemma \ref{lem:RStheta}, also in the skew-symmetric case $\varphi$ induces an order reversing bijection.

We summarize the above discussion in the following proposition:
\begin{proposition}\label{prop:RS-BOrbitsInvolutions}
The map
\[
\varphi:B_-\setminus G/ K\rightarrow  I^\epsilon, \quad B_-yK\mapsto (y\Omega'^{-1}y^t)T
\]
is well defined and gives an order reversing bijection.
\end{proposition}

Observe that if $y$ happens to be an element of $I^\epsilon$, it does not hold in general that $yT=(y\Omega'^{-1}y^t)T$, as the following example shows.

 \begin{example}
 Let \[
y=
\left[
\begin{array}{cccccc}
     1& 0&0&0&0&0\\
     0 &1&0&0&0&0\\
      0 &0&0&0&0&1\\
    0    &0&0&1&0&0\\
        0 &0&0&0&1&0\\
         0 &0&1&0&0&0
\end{array}
\right]
.
\]
Then \[
\varphi(B_-yK)=
(y\Omega_\epsilon y^t)T=y\Omega_\epsilon yT=\left[
\begin{array}{cccccc}
     0& 0&1&0&0&0\\
     0 &0&0&0&1&0\\
    \epsilon &0&0&0&0&0\\
    0 &0&0&0&0&1\\
    0 &\epsilon&0&0&0&0\\
    0 &0&0&\epsilon&0&0
\end{array}
\right]T
\]
and the latter coset corresponds to the permutation [351624].
 \end{example}

The next proposition is a straightforward consequence of the above material together with the $B_-$-equivariance of the open immersion $\iota_{\epsilon}:G/K\rightarrow \text{Mat}_\epsilon(2d,2d)$. Recall that the image of $\iota_\epsilon$ is the open set $\text{Mat}_\epsilon^\circ(2d,2d)$ of invertible matrices in $\text{Mat}_\epsilon(2d,2d)$. 

\begin{proposition}\label{prop:BOrbitsMatrixSpaceInvolutions}
Let $\mathcal{T}^\epsilon$ denote the set of $B_-$-orbit closures in $\text{Mat}_\epsilon^\circ(2d,2d)$, partially ordered by inclusion. There is an order reversing bijection between $\mathcal{T}^\epsilon$ and $I^\epsilon$.

\end{proposition}

\begin{proof}
    Recall the isomorphism $j_\epsilon$ from \eqref{eqn: IsoMatrixSpaceSymmSpace} . Being $B_-$-equivariant, it
 induces a poset isomorphism $\psi:\mathcal{T}^\epsilon\rightarrow B_-\setminus G/K$. By composing such an isomorphism with the order reversing bijection $\varphi$ from Proposition \ref{prop:RS-BOrbitsInvolutions} we obtain the desired order reversing bijection $\varphi\circ\psi: \mathcal{T}^\epsilon\rightarrow I^\epsilon$.\end{proof}

\begin{remark}
The above result will be applied in Section \ref{subsec:BorbitsAndRanks} to get rank conditions describing the poset of orbit closures in $\text{Mat}_\epsilon^\circ(2d,2d)$ (and hence in $G/K$) with respect to the action of a distinguished parabolic subgroup containing
$B_-$. Once these rank conditions are determined, they will be a central tool in Section \ref{sect:ZmapTotal} to relate the poset of orbit closures in a symmetric quiver orbit closure to the one for an appropriate parabolic action on a symmetric variety.
\end{remark}

\subsection{Permutation conventions}\label{subsec:permConventions}

Let $\mathfrak{S}_m$ denote the symmetric group on $m$ elements and let $v\in \mathfrak{S}_m$ be a permutation. We associate a permutation matrix to $v$ by placing $1$s at locations $(i,v(i))$, $1\leq i\leq m$ and $0$s elsewhere. Going forward, we will use $v$ to denote both the permutation in $\mathfrak{S}_m$ and its $m\times m$ permutation matrix. 

The \emph{Rothe diagram} $D(v)$ of a permutation matrix $v$ is the set $D(v)  = \{(i,j)\in [m]\times [m]\mid v(i)>j, v^{-1}(j)>i\}.$ This is equal to the set of locations $(i,j)\in [m]\times [m]$ left over after crossing out those cells which appear below a $1$ (and in the same column as the $1$) and to the right of a $1$ (and in the same row as the $1$), as well as the cell of the corresponding $1$ itself. The \emph{essential set} $\mathcal{E}(v)$ is defined by
 \[\mathcal{E}(v)=\{(i,j)\in D(v) \mid (i+1,j),(i,j+1)\not \in D(v)\}.\] 

Let $w\in \mathfrak{S}_{m}$ be an involution. Following Hamaker, Marberg, and Pawlowski in \cite{MR3846203}, 
define the \emph{orthogonal and symplectic Rothe diagrams} $D^+(w)$ and
$D^-(w)$ to be the subsets of positions in the usual Rothe diagram $D(w)$ that are weakly and strictly below the main diagonal, respectively. For $\epsilon\in\{\pm 1\}$ define the \textbf{essential sets} 
$\mathcal{E}^\epsilon(w)$ 
as follows:
 \[
 \mathcal{E}^{\epsilon}(w):=\{(i,j)\in D^{\epsilon}(w) \mid (i+1,j),(i,j+1)\not \in D^{\epsilon}(w)\}.
 \]

\begin{example}
If $w =  21563487\in\mathfrak{S}_8$
then the orthogonal and symplectic Rothe diagrams are the locations marked by circles and stars below
\[ D^+(w) =
 \left\{
\begin{array}{cccccccc}
\textcolor{red}{\star} & . & .& . & . & .&.&. 
\\
.& .& . & . & . & . &.&.
\\
. &. & \circ & . & . & . &.&.
\\
. & . & \circ & \textcolor{red}{\star} & . & . &.&.
\\
. & . & . & . & . & . &. &.
\\
. & . & . & . & . & .&.&. 
\\
. & . & . & . & . & . &\textcolor{red}{\star} &.
\\
. & . & . & . & . & .&.&. 
\end{array}
\right\},
\qquad
 D^{-}(w) = 
 \left\{
\begin{array}{cccccccc}
.& . & .& . & . & . &.&.
\\
. &. & . & . & . & . &.&.
\\
. & .& . & . & . & . &.&.
\\
. & . & \textcolor{red}{\star} & . & . & . &.&.
\\
. & . & . & . & . & . &.&.
\\
. & . & . & . & . & . &.&.
\\
. & . & . & . & . & . &.&.
\\
. & . & . & . & . & . &.&.
\end{array}
\right\}.
\]
The pairs $(i,j)$ belonging to the essential sets are marked by red stars.
\end{example}

\subsection{Orbit closures and rank conditions}\label{subsec:BorbitsAndRanks}
By Proposition \ref{prop:BOrbitsMatrixSpaceInvolutions} each 
$B_-$-orbit in $\text{Mat}_\epsilon^\circ(2d,2d)$, is represented uniquely by an involution in the set $I^\epsilon$.
For $y\in\Mat_\epsilon^\circ(2d,2d)$, we denote by $w_y\in I^\epsilon$ the unique involution corresponding to $\overline{B_-\cdot y}$. We notice that $w_y$ is the unique permutation matrix in the double coset $B_-yB\subseteq GL_{2d}$, so that for $\epsilon=-1$ it is not an element of $\Mat_\epsilon^\circ(2d,2d)$.

Let $M\in \Mat(2d,2d)$ and let $(i,j)$ be a pair of integers with $1\leq i,j\leq 2d$. 
If $M = (m_{s,t})_{1\leq s,t\leq 2d}$, define $M_{i\times j} = (m_{s,t})_{1\leq s\leq i, 1\leq t\leq j}$. That is, $M_{i\times j}$ is the northwest-justified $i\times j$ submatrix of $M$.
Define the associated northwest rank function by  $r_{i,j}(M) = \rank M_{i\times j}.$

\begin{proposition}\label{prop:matrixNWranks}
Let $x,y\in \Mat_\epsilon^\circ(2d,2d)$. Then,
\begin{enumerate}[(i)]
    \item $B_-\cdot x = B_-\cdot y$ if and only if $r_{i,j}(x)=r_{i,j}(y)$ for all $1\leq i,j\leq 2d$;
    \item $B_-\cdot x \subseteq \overline{B_-\cdot y}$  if and only if $r_{i,j}(x)\leq r_{i,j}(y)$   for all $1\leq i,j\leq 2d$.
\end{enumerate}
\end{proposition}
\begin{proof}

(ii) By Proposition \ref{prop:BOrbitsMatrixSpaceInvolutions},  for $x,y\in \text{Mat}^\circ_\epsilon(n,n)$, the inclusion $\overline{B_-\cdot x}\subseteq \overline{B_-\cdot y}$  (and hence $B_-\cdot x\subseteq \overline{B_-\cdot y}$) holds if and only if the corresponding elements $w_x, w_y\in I^\epsilon$ satisfy $w_x\geq w_y$ in Bruhat order. By \cite{Ful92}, $w_x\geq w_y$ in Bruhat order if and only if $r_{i,j}(x)\leq r_{i,j}(y)$ for all $i,j$.

(i) We have 
$B_-\cdot x = B_-\cdot y \Leftrightarrow \overline{B_-\cdot x} = \overline{B_-\cdot y} \Leftrightarrow$
both  
$B_-\cdot x \subseteq \overline{B_-\cdot y}$
and
$B_-\cdot y \subseteq \overline{B_-\cdot x}$.
Then by (ii) this is equivalent to both $r_{i,j}(x)\leq r_{i,j}(y)$ and $r_{i,j}(y)\leq r_{i,j}(x)$ for all $i,j$.
\end{proof}

Using the orthogonal and symplectic essential sets, we can further  decrease the number of rank functions needed to characterize orbit closures.

\begin{proposition}\label{prop:matrixNWranks2}
Let $x,y\in \Mat_\epsilon^\circ(2d,2d)$. Then,
 $B_-\cdot x \subseteq \overline{B_-\cdot y}$
 if and only if $r_{i,j}(x)\leq r_{i,j}(y)$ for all $(i,j) \in \mathcal{E}^\epsilon(w_y)$.
\end{proposition}

\begin{proof}
By Proposition \ref{prop:matrixNWranks} (ii), we can characterize an orbit closure as
\begin{equation}
    \overline{B_-\cdot y} = \{z\in \text{Mat}^\circ_\epsilon(2d,2d) \mid r_{i,j}(z)\leq r_{i,j}(w_y) \text{ for all }(i,j)\}.
\end{equation}
Then using \cite[Prop. 2.16]{MP19}, we have
\begin{equation}\label{eq:MPranks}
        \overline{B_-\cdot y} = \{z\in \text{Mat}^\circ_\epsilon(2d,2d) \mid r_{i,j}(z)\leq r_{i,j}(w_y) \text{ for all }(i,j) \in \mathcal{E}^\epsilon(w_y)\}.
\end{equation}
Now, $B_-\cdot x\subseteq \overline{B_-\cdot y}$ if and only if $x\in \overline{B_-\cdot y}$.  By equation \eqref{eq:MPranks}, $x\in \overline{B_-\cdot y}$ if and only if $r_{i,j}(x)\leq r_{i,j}(w_y) = r_{i,j}(y)$ for all $(i,j) \in \mathcal{E}^\epsilon(w_y)$. 
\end{proof}

We now turn our attention to orbits of parabolic subgroups.
Given a series of integers $(d_1, \dotsc, d_r)$ such that $\sum_{i=1}^r d_i = 2d$, we consider a block structure on $2d\times 2d$ matrices where the sizes of the column blocks are $d_1,\dotsc,d_r$ moving from left to right, and the same for rows moving from top to bottom.
We have an associated parabolic subgroup $P \leq G=GL(2d)$ consisting of block lower triangular matrices with square blocks of size $d_1,\dotsc,d_r$ from top-left to bottom-right along the main diagonal.
We write $\mathcal{C}=\{\sum_{i=1}^k d_i \mid 1 \leq k \leq 2d\}$ for the set of partial sums of the sequence $(d_1, \dotsc, d_r)$.  These are the indices of the rows (and columns) of the lower right corners of the blocks.

\begin{proposition}\label{prop:POrbitNWRanks}
Let $x, y \in \Mat_\epsilon^\circ(2d,2d)$. Then,
\begin{enumerate}[(i)]
    \item $P\cdot x = P\cdot y$ if and only if $r_{i,j}(x)=r_{i,j}(y)$ for all $i,j \in \mathcal{C}$;
    \item $P\cdot x \subseteq \overline{P\cdot y}$ if and only if $r_{i,j}(x)\leq r_{i,j}(y)$ for all $i,j \in \mathcal{C}$.
\end{enumerate}
\end{proposition}
\begin{proof}
(ii) For each $P$-orbit $P\cdot x$, there is a unique maximal
$B_-$-orbit contained in it, and we denote a representative of this orbit by $x'$ in this proof.
This representative satisfies $\overline{B_-\cdot x'} = \overline{P\cdot x}$. By Proposition \ref{prop:matrixNWranks2}, we have 
$\overline{ B_-\cdot x'}\subseteq \overline{B_-\cdot y'}$ if and only if $r_{i,j}(x')\leq r_{i,j}(y')$ for all $(i,j)\in  \mathcal{E}^\epsilon(w_y')$. 
Now we claim that because $\overline{B_-\cdot y'}$ is a $P$-orbit closure, 
we have a containment of index sets $ \mathcal{E}^\epsilon(w_y') \subseteq \mathcal{C} \times \mathcal{C}$.
Assume for contradiction this is not true, say $(i,j)\in \mathcal{E}^\epsilon(w_y')$ but not in $\mathcal{C} \times \mathcal{C}$, specifically that $i \not \in \mathcal{C}$ (the case $j \not\in \mathcal{C}$ is similar). 
Then $(i,j)$ is in the orthogonal or symplectic Rothe diagram of $w'_y$ but $(i+1,j)$ is not in the diagram, implying that $r_{i,j}(w'_y)<r_{i+1,j}(w'_y)$.
Furthermore, the matrix $w'_y$ has a 1 in row $i+1$ and some column $k$ with $1 \leq k \leq j$, and has all 0s in row $i$ and these same columns. Recall that if $\epsilon=-1$ the permutation matrix $w'_y$ is not an element of $\Mat_\epsilon^\circ(2d,2d)$. Thus, we denote by $[w_y']_\epsilon$ the  matrix obtained from $w_y'$ by multiplying all entries (stricly) under the diagonal by $\epsilon$ so that now $[w_y']_\epsilon\in\Mat_\epsilon^\circ(2d,2d)$ and clearly $r_{i,j}(w_y')=r_{i,j}([w_y']_\epsilon)$.

By definition of orthogonal and symplectic diagrams, we have $i \geq j$; we first treat the case $i > j$.
Since $i \not\in \mathcal{C}$, the parabolic subgroup $P$ contains the matrix $E$ obtained by  switching rows $i$ and $i+1$ of the identity matrix, so $E\cdot [w'_y]_\epsilon \in P\cdot [w'_y]_\epsilon$. 
Because of the location of the 1 in row $i+1$ of $w'_y$, the matrix 
$E\cdot [w'_y]_\epsilon= E[w'_y]_\epsilon E^t$ has the $\epsilon 1$ from row $i+1$ of $[w'_y]_\epsilon$ moved up to row $i$, and switching columns $i$ and $i+1$ does not take any $\epsilon 1$s out of the northwest $i \times j$ submatrix
because of our assumption that $i > j$.
Thus we get $r_{i,j}(E\cdot [w'_y]_\epsilon)=r_{i+1,j}([w'_y]_\epsilon)=r_{i+1,j}(w'_y) > r_{i,j}(w'_y)$, contradicting maximality of the $B$-orbit of $w'_x$ in its $P$-orbit.

This completes the proof for $\epsilon = -1$ as $i>j$ in this setting.

When $\epsilon = +1$, we have
$w'_y=[w_y']_\epsilon$ and we must also consider the case $i=j$.  
The case of $i=j$ is similar to the previous one, 
except that if $w'_y$ has 1s in entries $(i+1,i)$ and $(i,i+1)$, we get $E\cdot w'_y=w'_y$ so the argument above does not work.  One must instead use the matrix $F$ defined as the direct sum of matrices below: 
\begin{equation}
F:= I_{i-1}\oplus \begin{bmatrix} \frac{1}{\sqrt{2}} &\frac{1}{\sqrt{2}} \\ \frac{-1}{\sqrt{2}} & \frac{1}{\sqrt{2}} \end{bmatrix} \oplus I_{2d-i-1}
\end{equation}
Then the computation 
\begin{equation}
    \begin{bmatrix} \frac{1}{\sqrt{2}} &\frac{1}{\sqrt{2}} \\ \frac{-1}{\sqrt{2}} & \frac{1}{\sqrt{2}} \end{bmatrix} \cdot 
    \begin{bmatrix} 0 & 1\\ 1 & 0 \end{bmatrix} = \begin{bmatrix} \frac{1}{\sqrt{2}} &\frac{1}{\sqrt{2}} \\ \frac{-1}{\sqrt{2}} & \frac{1}{\sqrt{2}} \end{bmatrix} \begin{bmatrix} 0 & 1\\ 1 & 0 \end{bmatrix} \begin{bmatrix} \frac{1}{\sqrt{2}} &\frac{1}{\sqrt{2}} \\ \frac{-1}{\sqrt{2}} & \frac{1}{\sqrt{2}} \end{bmatrix}^t =
    \begin{bmatrix} 1 & 0\\ 0 & -1 \end{bmatrix}
\end{equation}
shows that $r_{i,i}(F\cdot w'_y)=r_{i+1,i}(w'_y) > r_{i,i}(w'_y)$ and argument is completed as in the previous case.

(i) 
We have $P\cdot x = P\cdot y \Leftrightarrow \overline{P\cdot x} = \overline{P\cdot y} \Leftrightarrow$ both  $P\cdot x \subseteq \overline{P\cdot y}$ and $P\cdot y \subseteq \overline{P\cdot x}$.
Then by (ii) this is equivalent to both $r_{i,j}(x)\leq r_{i,j}(y)$ and $r_{i,j}(y)\leq r_{i,j}(x)$ for all $i,j \in \mathcal{C}$.
\end{proof}

\section{Identifying symmetric bipartite type \texorpdfstring{$A$}{Lg} quiver orbit closures with Mars-Springer varieties}\label{sect:ZmapTotal}

Throughout this section, let $Q$ be a symmetric type $A$ quiver with bipartite orientation, so that every vertex is either a source or a sink, and without loss of generality we can   label $Q$ as in \eqref{eq:typea}. 
\begin{equation}\label{eq:typea}
\vcenter{\hbox{\begin{tikzpicture}[point/.style={shape=circle,fill=black,scale=.5pt,outer sep=3pt},>=latex]
   \node[outer sep=-2pt] (y0) at (-1,0) {$y_0$};
   \node[outer sep=-2pt] (x1) at (0,1) {$x_1$};
  \node[outer sep=-2pt] (y1) at (1,0) {$y_1$};
   \node[outer sep=-2pt] (x2) at (2,1) {$x_2$};
  \node[outer sep=-2pt] (y2) at (3,0) {$y_2$};
  \node (dots) at (3.75, 0.5) {$\cdots\cdots$};
   \node[outer sep=-2pt] (x3) at (5,1) {$x_{n-1}$};
  \node[outer sep=-2pt] (y3) at (6,0) {$y_{n-1}$};
  \node[outer sep=-2pt] (x4) at (7,1) {$x_n$};
  \path[->]
  	(x1) edge node[left] {$\alpha_1$} (y0) 
	(x1) edge node[left] {$\beta_1$} (y1)
  	(x2) edge node[left] {$\alpha_2$} (y1) 
	(x2) edge node[left] {$\beta_2$} (y2)
	(x3) edge node[left] {$\beta_{n-1}$} (y3)
  	(x4) edge node[left] {$\alpha_n$} (y3); 
   \end{tikzpicture}}} .
\end{equation}
Let $\bd$ be a dimension vector for $Q$ and let $d_x = \sum_{i=1}^n\bd(x_i)$ and $d_y = \sum_{i=0}^{n-1}\bd(y_i)$. 
For $\bd$ a symmetric dimension vector, we necessarily have that $d_x = d_y$.

In this section, we show that the rank conditions which characterize symmetric quiver orbit closures (see Section \ref{ssec:RankConditions}) are equivalent to those that define Mars-Springer varieties from Section \ref{subsec:MSslices}  (see Sections \ref{ssec:symmetricZmap} and \ref{sect:quiverP}). 

\subsection{Rank conditions determining orbits in $\srep_Q(\bd)$}\label{ssec:RankConditions}
We now show how orbits in $\srep_Q(\bd)$ are determined by rank conditions.
Following \cite[\S3.1]{KR15}, we call a connected subquiver $J$ of $Q$ an \emph{interval} in $Q$, and denote this by $J\subseteq Q$.
To each interval $J \subset Q$, we have an associated rank function $r_J \colon \rep_Q(\bd) \to \ZZ_{\geq 0}$ defined as follows.  Given $V \in \rep_Q(\bd)$, form the matrix
\begin{equation}\label{eq:VQ}
V_Q:= \begin{bmatrix}
& & & & V_{\alpha_1}\\
& & & V_{\alpha_2} & V_{\beta_1}\\
& & V_{\alpha_3} & V_{\beta_2}\\
& \Ddots&  \Ddots & \\
V_{\alpha_n} & V_{\beta_{n-1}} \\
\end{bmatrix},
\end{equation}
then let $r_J(V)$ be the rank of the unique submatrix containing maps associated exactly to the arrows in the interval $J$.
Each $r_J$ is constant on $GL(\bd)$-orbits in $\rep_Q(\bd)$.
Using the embedding \eqref{eq:sreptorep}, we may also use the above notations with symmetric representations, and we have the following as a direct consequence of Theorem \ref{thm:DWorbits} and \cite[Prop.~3.1]{KR15}.

\begin{proposition}\label{prop:quiverranksorbits}
    For $V, W \in \srep_Q(\bd)$, we have $W \in G(\bd)\cdot V$ if and only if $r_J(W)=r_J(V)$ for all intervals $J \subseteq Q$.
\end{proposition}

The following proposition gives rank function criteria for an ordinary representation of $Q$ to be isomorphic to a symmetric representation of that quiver, recalling from Theorem \ref{thm:DWorbits} that this isomorphism can equivalently be interpreted as representations of $Q$ or $Q^\epsilon$. It is not used until Section \ref{sect:mainZperm2}.

\begin{proposition}\label{prop:regtosym} Let $\bd$ be a symmetric dimension vector for the bipartite type $A$ quiver $Q$.
Let $V \in \rep_Q(\bd)$ be a representation such that $r_J(V) = r_{\tau J}(V)$ for all intervals $J \subseteq Q$.
Then,
\begin{enumerate}[(i)]
    \item there exists $W \in \srep_Q(\bd)$, with $\epsilon = +1$, such that $W \in GL(\bd)\cdot V$ with respect to the embedding \eqref{eq:sreptorep};
    \item if, additionally, $r_J(V)$ is even for each interval $J$ that satisfies $J = \tau J$, then there exists $W \in \srep_Q(\bd)$, with $\epsilon = -1$, such that $W \in GL(\bd)\cdot V$ with respect to the embedding \eqref{eq:sreptorep}.
\end{enumerate}
Furthermore, in both cases the representation $W$ is unique up to isomorphism.
\end{proposition}
\begin{proof}
We consider the Krull-Schmidt decomposition
\[
V\cong \bigoplus_{J\subseteq Q} s_J \mathbb{I}_J,
\]
where $s_J \in \ZZ_{\geq 0}$ and $\mathbb{I}_J$ denotes the indecomposable representation supported on the interval $J\subseteq Q$.
To prove the statement, it suffices to show that $V$ can be interpreted as symmetric representation of $Q^\epsilon$ as in Section \ref{sec:symreps}, since each such symmetric representation corresponds to a point of $\srep_Q(\bd)$ by choosing a basis.  Since a direct sum of symmetric representations is symmetric, we can distribute over direct sums at our convenience. 

For intervals $J$ such that $\tau J \neq J$, one can check directly that $\mathbb{I}_J \oplus \mathbb{I}_{\tau J}$ is symmetric.
It follows from the rank-to-lace formula \cite[Lemma 3.4]{KR15} that $s_J = s_{\tau J}$ for all intervals $J\subseteq Q$, 
so also $s_J \mathbb{I}_J \oplus s_{\tau J} \mathbb{I}_{\tau J} = s_J (\mathbb{I}_J \oplus \mathbb{\tau J})$ is symmetric.

Now consider intervals $J$ such that $\tau J = J$. 
To complete the proof of (i),  
one can directly check that each summand $\mathbb{I}_J$ is already symmetric, when $\epsilon = +1$. 

To complete the proof of (ii), assume that  $r_J(V)$ is even for each interval $J$ such that $J = \tau J$. Applying the rank-to-lace formula \cite[Lemma 3.4]{KR15}, we conclude 
that $s_J$ is even for each interval $J$ such that $J = \tau J$. 
Then one can directly check that $\mathbb{I}_J \oplus \mathbb{I}_J$ is symmetric, when $\epsilon = -1$, so $s_J\mathbb{I}_J$ is symmetric.
Having covered all intervals $J \subseteq Q$, we see that $V$ is a direct sum of symmetric representations and thus symmetric itself.

Now the furthermore part follows because if $W'$ is another symmetric representation in the same $GL(\bd)$-orbit as $V$, then it is also in the same $G(\bd)$-orbit by Theorem \ref{thm:DWorbits}.
\end{proof}

\subsection{The bipartite type $A$ Zelevinsky map}\label{sec:ordinaryZmap}

In this subsection, we recall the bipartite type $A$ Zelevinsky map from \cite{KR15}. Let $Q$ be a bipartite type $A$ quiver of the form \eqref{eq:typea} and $\rep_Q(\bd)$ a representation space for $Q$. Let $d_x = \sum_{i = 1}^n \bd(x_i)$ and $d_y = \sum_{i=0}^{n-1}\bd(y_i)$, and define the space of matrices $Y_\bd$ as
\begin{equation}
Y_\bd:= \left\{\begin{bmatrix} X & I_{d_y}\\I_{d_x}&0\end{bmatrix} : X\in \text{Mat}(d_y, d_x) \right\},
\end{equation}
where $I_{d_y}$ and $I_{d_x}$ are $d_y \times d_y$ and $d_x\times d_x$ identity matrices. 
Define the \emph{bipartite Zelevinsky map} as
\begin{equation}\label{eq:bipartiteZmap}
    \zeta\colon\rep_Q(\bd)\hookrightarrow Y_{\bd}, \quad \zeta(V) = \begin{bmatrix} V_Q &I_{d_y} \\ I_{d_x}&0 \end{bmatrix},
\end{equation}
where $V_Q$ is as in \eqref{eq:VQ}. 
There is an associated retraction
\begin{equation}\label{eq:piretract}
\pi\colon Y_\bd \to \rep_Q(\bd), \quad \pi\circ \zeta=id_{\rep_Q(\bd)}
\end{equation}
obtained by extracting submatrices as in \eqref{eq:VQ} and assigning each to the corresponding arrow of $Q$.
There is also an associated embedding
\begin{equation}\label{eq:GLdtoP}
    \xi\colon GL(\bd) \hookrightarrow GL(d_x+d_y)
\end{equation}
sending a collection $(g_z)_{z \in Q_0}$ to the block diagonal matrix with $g_{y_0}, \dotsc, g_{y_{n-1}}, \dotsc, g_{x_n}, \dotsc, g_{x_1}$ along the diagonal from top-left to bottom-right.  The embedding $\zeta$ of \eqref{eq:bipartiteZmap} is equivariant with respect to this embedding of groups.

Using the bipartite Zelevinsky map, one obtains an identification between the collection of rank conditions which characterize the $GL(\bd)$-orbit through $V\in \text{rep}_Q(\bd)$ (see \cite[Prop.~3.1]{KR15}) and a collection of ranks of northwest justified submatrices of $\zeta(V)$. This identification is an intermediate step in the proof of the main theorem of \cite{KR15}, which says that each type $A$ quiver orbit closure is isomorphic, up to a smooth factor, to an open subvariety of a Schubert variety in a type $A$ flag variety. In the present paper, we also use this identification of rank conditions as an intermediate step in the proof of our main theorem  (Theorem \ref{thm:mainTheorem}).

\subsection{Symmetric Zelevinsky map}\label{ssec:symmetricZmap} 
In this section, we define a symmetric quiver analog of the bipartite Zelevinsky map recalled in the previous subsection. Let $Q^\epsilon$ be a symmetric bipartite quiver labeled as in \eqref{eq:typea}, where $\epsilon=s(a)$ with $a$ the unique arrow fixed by the involution of $Q$.  Fix a symmetric dimension vector $\bd$.
Let $d = d_x = d_y$ and consider the $2d\times 2d$ matrix $x = \begin{bmatrix}0&J\\I &0 \end{bmatrix}$.
With $G=GL(2d)$ and $K$ as in Section \ref{subsec:Conventions} using the specific form $\Omega_\epsilon$ in \eqref{eq:omegapm}, so always either an orthogonal or symplectic subgroup,
we let $S=S_{\overline{x}}\subseteq G/K$ be the corresponding Mars-Springer slice as in \eqref{eq:ourMSslice}. By Lemma \ref{lem:MSSIsomorphism}, we identify this Mars-Springer slice with the matrix space
\begin{equation}\label{eq:MatrixMSSlice}
 Y^\epsilon_d = 
\left\{
\begin{bmatrix}
\Sigma& I_d\\ \epsilon I_d &0\end{bmatrix} \mid \Sigma \in \Mat_\epsilon(d, d) \right\}
\end{equation}
in \eqref{eqn:MatrixSlice} via the embedding \eqref{eqn:MPembedding}.

We also define a modification $V_Q^\epsilon$ which is the same as the matrix $V_Q$ in \eqref{eq:VQ}, except that lower-left $n-1$ submatrices $V_{\alpha_n}, V_{\beta_{n-1}},\dotsc$
are multiplied by the sign $\epsilon$.
We make the following easy but important observation:
\begin{lemma}\label{lem:symmNW}
If $V\in \srep_Q(\bd)$, then $V^\epsilon_Q\in \Mat_\epsilon(d,d)$.
\end{lemma}

In light of Lemma \ref{lem:symmNW}, we can define the \textbf{symmetric Zelevinsky map} to be
\begin{equation}\label{eq:zetaKdef}
    \zeta^\epsilon:\text{srep}_Q(\bd)\rightarrow Y^\epsilon_d,  \quad \zeta^\epsilon(V) = \begin{bmatrix} V^\epsilon_Q &I \\ \epsilon I_d &0 \end{bmatrix}.
\end{equation}

The symmetric Zelevinsky map is closely related to the bipartite Zelevinsky map. In particular, since each $V\in \text{srep}_Q(\bd)$ is also an ordinary representation of a bipartite type $A$ quiver via the embedding \eqref{eq:sreptorep}, we may apply both $\zeta$ and $\zeta^\epsilon$ to $V$, and
it is easy to see that
$r_{i,j}(\zeta^\epsilon(V)) = r_{i,j}(\zeta(V))$
for all $1\leq i,j\leq 2n$.
Now recalling the definition of $\mathcal{C}$ directly before Proposition \ref{prop:POrbitNWRanks} 
with respect to the sequence
\begin{equation}\label{eq:blocksizes}
    ({\bf{d}}(y_0), {\bf{d}}(y_1), \dots, {\bf{d}}(y_{n-1}), {\bf{d}}(x_{n}),\dots, {\bf{d}}(x_2),{\bf{d}}(x_1)),
\end{equation} 
we get the following immediately from \cite[Lem.~A.3]{KR15} and its proof.

\begin{lemma}\label{lem:NWandquiverranks}
Let $V,W \in \srep_Q(\bd)$.  Then $r_J(V)=r_J(W)$ for all 
$J$ if and only if $r_{i,j}(\zeta^\epsilon(V))=r_{i,j}(\zeta^\epsilon(W))$ for all $i,j \in \mathcal{C}$. The analogous statement holds upon replacing $=$ with $\leq$.
\end{lemma}

\subsection{Identifying quiver orbits with $P$-orbits}\label{sect:quiverP}
Let $P \leq GL(2d)$ to be the parabolic subgroup of block lower-triangular matrices with blocks of sizes given by \eqref{eq:blocksizes} along the diagonal from northwest to southeast.
As the embedding \eqref{eq:sreptorep} is equivariant with respect to \eqref{eq:GdGLd}, we find that $\zeta^\epsilon$ is equivariant with respect to 
\begin{equation}\label{eq:GdtoGL}
\bar{\xi}\colon G(\bd) \hookrightarrow P \leq GL(2d)
\end{equation}
obtained from composing \eqref{eq:GLdtoP} with \eqref{eq:GdGLd} and noting that the target lies in $P$.
The following lemma strengthens this to say that distinct $G(\bd)$-orbits in $\srep_Q(\bd)$ are sent to distinct $P$-orbits in
$\Mat_\epsilon^\circ(2d,2d)$.

\begin{lemma}\label{lem:PtoGdorbits}
Let $V, W \in \srep_Q(\bd)$ such that $\zeta^\epsilon(V)$ and $\zeta^\epsilon(W)$ lie in the same $P$-orbit.  Then $V,W$ lie in the same $G(\bd)$-orbit of $\srep_Q(\bd)$.
\end{lemma}
\begin{proof}
By Proposition \ref{prop:POrbitNWRanks}, we have $r_{i,j}(\zeta^\epsilon(V))=r_{i,j}(\zeta^\epsilon(W))$ for all $(i,j) \in \mathcal{C}$.
Then from Lemma \ref{lem:NWandquiverranks} we get that $r_J(V)=r_J(W)$ for all  intervals $J$ in $Q$, which implies that $V,W$ are in the same $G(\bd)$-orbit by Proposition \ref{prop:quiverranksorbits}.
\end{proof}

Recall that when we denote an orbit by $\mathscr{O}^\circ$, we denote it closure by $\mathscr{O}$.

\begin{proposition}\label{prop:POrbitimage}
Let $O^\circ \subset 
\Mat_\epsilon^\circ(2d,2d)$ be a $P$-orbit such that $O^\circ \cap \im \zeta^\epsilon \neq \emptyset$.
Then $O^\circ \cap 
Y_d^\epsilon= \zeta^\epsilon(\mathscr{O}^\circ)$ for some orbit $\mathscr{O}^\circ \subset \srep_Q(\bd)$, where the intersection is taken as varieties.
Furthermore, we have $\zeta^\epsilon(W) \in O \cap 
Y_d^\epsilon$ for each $W \in \mathscr{O}$.
\end{proposition}
\begin{proof}
Since $O^\circ \cap \im \zeta^\epsilon \neq \emptyset$, there exists $V \in \srep_Q(\bd)$ such that $\zeta^\epsilon(V) \in O^\circ$.
Let $\mathscr{O}^\circ=G(\bd)\cdot V$.
Since $\zeta^\epsilon$ is equivariant with respect to the embedding $G(\bd)\hookrightarrow P$, and $\im \zeta^\epsilon \subseteq
Y_d^\epsilon$, we get $\zeta^\epsilon(\mathscr{O}^\circ) \subseteq O^\circ \cap 
Y_d^\epsilon
$.
For the reverse containment, let $x \in O^\circ \cap 
Y_d^\epsilon$. Since $x$ and $\zeta^\epsilon(V)$ lie in the same $P$-orbit,
we have $r_{i,j}(x)=r_{i,j}(\zeta^\epsilon(V))$ for all $i,j \in \mathcal{C}$ by Proposition \ref{prop:POrbitNWRanks}(i).
In particular, $x$ satisfies the particular subset of rank conditions which define the image of $\zeta^\epsilon$, as in \cite[Lem.~A.2]{KR15}.
So $x=\zeta^\epsilon(V')$ for some $V' \in \srep_Q(\bd)$.
Then Lemma \ref{lem:NWandquiverranks} gives $r_J(V)=r_J(V')$ for all intervals $J \subseteq Q$, so Proposition \ref{prop:quiverranksorbits} implies that 
$V'\in \mathscr{O}^\circ$ as well.  Thus, $x =\zeta^\epsilon(V') \in \zeta^\epsilon(\mathscr{O}^\circ)$, completing the proof.
The final statement passing to closures follows from the fact that $\zeta^\epsilon$ is a closed embedding.
\end{proof}

For $V \in \srep_Q(\bd)$, the following subvariety of $\srep_Q(\bd)$ naturally arises when considering rank conditions.  Subvarieties of representation spaces defined like this are not generally irreducible, though it will turn out that they are in this special case of symmetric type $A$ representations.
\begin{equation}\label{eq:XV}
    \mathscr{X}_V := \{ W \in \srep_Q(\bd) \mid r_J(W) \leq r_J(V) \text{ for all }J\}
\end{equation}

\begin{lemma}\label{lem:sliceisorankvariety}
Let $V \in \srep_Q(\bd)$ and $O=\ol{P\cdot \zeta^\epsilon(V)}$.  For $W \in \srep_Q(\bd)$, we have that $W \in \mathscr{X}_V$ if and only if $\zeta^\epsilon(W) \in O \cap 
Y_d^\epsilon$.
\end{lemma}
\begin{proof}
We have $W \in \mathscr{X}_V$ if and only if $r_J(W) \leq r_J(V)$ by definition, and the latter is equivalent to $r_{i,j}(\zeta^\epsilon(W))\leq r_{i,j}(\zeta^\epsilon(V))$ for all $(i,j)\in \mathcal{C}$ by Lemma \ref{lem:NWandquiverranks}.
Now applying Proposition \ref{prop:POrbitNWRanks}, noting that $\im \zeta^\epsilon \subseteq 
Y_d^\epsilon$, shows these are equivalent to $\zeta^\epsilon(W) \in O \cap 
Y_d^\epsilon$.
\end{proof}

\begin{lemma}\label{lem:orbitcapSirred}
Let $O \subset 
\Mat_\epsilon^\circ(2d,2d)$ be a $P$-orbit closure such that $O^\circ \cap \im \zeta^\epsilon \neq \emptyset$. Then $O \cap 
Y_d^\epsilon$ is a reduced and irreducible scheme.
\end{lemma}
\begin{proof}
By Lemma \ref{lem:MSSIsomorphism}  we know that $O \cap 
Y_d^\epsilon$ is a reduced scheme.
Let $V \in \srep_Q(\bd)$ be such that $x:=\zeta^\epsilon(V) \in O^\circ$.
By Lemma \ref{lem:sliceisorankvariety}, it is equivalent to show $\mathscr{X}_V$ is an irreducible subscheme of $\srep_Q(\bd)$.

We claim $\mathscr{X}_V= \ol{G(\bd)\cdot V}$, and is thus reduced and irreducible. Since each $r_J$ is semicontinuous, we get $\mathscr{X}_V\supseteq \ol{G(\bd)\cdot V}$, as is always the case for subvarieties of representation varieties defined like this.
To show the reverse containment, since $\mathscr{X}_V$ is a $G(\bd)$-equivariant set, it is enough to show that every $G(\bd)$-orbit in $\mathscr{X}_V$ is also contained in $\ol{G(\bd)\cdot V}$.
So let $W \in \srep_Q(\bd)$ satisfy $G(\bd)\cdot W \subseteq \mathscr{X}_V$.
Then $r_J(W) \leq r_J(V)$ for all intervals $J$ by definition of $\mathscr{X}_V$.
Now considering $V,W$ as representations of the underlying type $A$ quiver without symmetric data via \eqref{eq:sreptorep},
it is well-established (for example in \cite{AdF85}) that this implies
$GL(\bd)\cdot W \subseteq \ol{GL(\bd)\cdot V}$.

To work our way back to the symmetric setting, 
the first equivalence of \cite[Thm.~7.1]{BC21} shows that in our type $A$ quiver setting,
$GL(\bd)\cdot W \subseteq \ol{GL(\bd)\cdot V}$
is equivalent to $G(\bd)\cdot W \subseteq \ol{G(\bd)\cdot V}$.
So varying over all $W$ such that $G(\bd)\cdot W \subseteq \mathscr{X}_V$ proves the claim and overall lemma.
\end{proof}

\begin{remark}
In the general setting of Section \ref{sec:symvarieties},
it turns out that Mars-Springer slices are always irreducible when $K$ is symplectic,
but can be reducible when $K$ is orthogonal (see \cite[\S2B.~Remark]{WWY}).
Thus our proof that detours through a citation to difficult work on degenerations of symmetric quiver representations 
is only needed to cover the case where $\epsilon=1$.
When $\epsilon=-1$, Lemma \ref{lem:orbitcapSirred} is immediate from the fact that every Mars-Springer slice for a symmetric pair $(GL(2d), Sp(2d))$ is irreducible.

Since our proof shows that the particular Mars-Springer slices in Lemma \ref{lem:orbitcapSirred} are always irreducible, there must be something special about these.  We pose it as a problem to the reader to find a direct proof of irreducibility for these particular Mars-Springer slices, in order to make the methods of this paper more self-contained.
\end{remark}

\section{From arbitrary orientation to bipartite orientation}\label{sec:orientation}
In this section we give a brief review of homogeneous fiber bundles, then use this to show that any symmetric representation variety for type $A$ quiver of arbitrary orientation is related to one for the bipartite orientation by a homogeneous fiber bundle construction.

\subsection{Homogeneous fiber bundles}\label{sect:bundles}
Let $G$ be an algebraic group, $H\leq G$ a closed and connected algebraic subgroup, and $X$ a quasiprojective $H$-variety.
Write $G *_H X$ for the quotient of $G \times X$ by the free left action of $H$ given by $h\cdot (g,x) = (gh^{-1}, h\cdot x )$. 
This quotient, called an \emph{induced space} or \emph{homogeneous fiber bundle}, is a $G$-variety for the action $g\cdot (g'*x) = gg'*x$.
One may consult \cite[\S5.2.18]{CGbook} or \cite[\S14.6]{FSR17} as general references, for example.
The following lemma gives a criterion for a homogeneous fiber bundle to decompose as a product of varieties.

\begin{lemma}\label{lem:KXHG}\cite[Lemma 14.6.9]{FSR17}
Let $G$ be an affine algebraic group and $H\leq G$ a closed subgroup. Let $X$ be a $G$-variety (where the $G$-action extends the $H$-action) and endow $G/H \times X$ with the diagonal left $G$-action. Then there exists a $G$-equivariant isomorphism $\Gamma: G *_H X \rightarrow  G/H \times X$. 
\end{lemma}

The $X$, $H$, $G$ in this paper always satisfy the hypotheses of \cite[Thm.~14.6.5]{FSR17}, so that $G\times X \to G *_H X$ is a geometric quotient, and the map $X \to G *_H X$ sending $x \mapsto 1*x$ is a closed immersion.  In this way, we identify $X \subseteq G *_H X$ when taking intersections of $X$ with subvarieties of $G *_H X$.
The following standard results connect the problems we study for an $H$-variety $X$ to those for the $G$ variety $G *_H X$ (see \cite[\S2.3]{KR21} for more detail).

\begin{proposition}\label{prop:fiberbundle}
In the general setup of this section, all of the following hold.
\begin{enumerate}[(i)]
\item The following association is a bijection.  
\begin{equation}\label{eq:GHvarieties}
\begin{split}
\left\{\begin{tabular}{c} $G$-stable subvarieties\\ of $G *_H X$\end{tabular} \right\}
& \to
\left\{\begin{tabular}{c}  $H$-stable subvarieties\\ of $X$ \end{tabular} \right\}
\\
Y \qquad \qquad &\mapsto \qquad \qquad Y \cap X\\
G\cdot Z \qquad \qquad &\mapsfrom \qquad \qquad Z
\end{split}
\end{equation}
In particular, it restricts to a bijection on orbits and induces an isomorphism of orbit closure posets.
\item Let $Z \subseteq X$ be an $H$-stable subvariety. Any smooth equivalence class of singularity that occurs in $Z$ also occurs in $G\cdot Z$. 
\item Restriction along $X \hookrightarrow G *_H X$ induces an isomorphism of equivariant Grothendieck groups $K_G(G *_H X) \simeq K_H(X)$.
\end{enumerate}
\end{proposition}

\subsection{Reduction the bipartite setting} \label{sec:bipartitereduction}
Proposition \ref{prop:bipartitereduction} is the main result of this section, which will reduce the proof of our main theorem to the case that $Q$ is bipartite, using general properties of homogeneous fiber bundles.

\begin{proposition}\label{prop:bipartitereduction}
Let $Q$ be a symmetric type $A$ quiver with dimension vector $\bd$.
Then there exists a symmetric bipartite type $A$ quiver  $\tilde{Q}$ with dimension vector $\tilde{\bd}$, along with an open subvariety $\mathscr{U} \subseteq \srep_{\tilde{Q}}(\tilde{\bd})$
admitting a $G(\tilde{\bd})$-equivariant isomorphism
\begin{equation}\label{eq:Ubundle}
    \mathscr{U} \simeq G(\tilde{\bd}) *_{G(\bd)} \srep_Q(\bd).
\end{equation}
\end{proposition}

There are two possibilities for $Q$ which we will need to consider separately in the proof: when $n$ is odd, $\tau$ fixes the middle vertex, and when $n$ is even, $\tau$ fixes the middle arrow.
In each case, we use the formalism of arrow contraction from \cite[\S2.5]{KR21}, which is greatly simplified in the type $A$ quiver setting, but needs to be extended to symmetric quivers.

We fix the following conventions: let $Q$ be a bipartite type $A$ quiver.  Define a new type $A$ quiver $\tilde{Q}$ by taking each length 2 path in $Q$ and inserting an arrow in the middle so that the resulting 3 arrows do not form a path.  Locally, splitting a vertex labeled $\star$ looks like this:
\begin{equation}\label{eq:splitex}
\vcenter{\hbox{\begin{tikzpicture}[point/.style={shape=circle,fill=black,scale=.5pt,outer sep=3pt},>=latex]
   \node[outer sep=-2pt] (y) at (0,0) {$x$};
  \node[outer sep=-2pt] (x) at (2,0) {$\star$};
  \node[outer sep=-2pt] (ty) at (4,0) {$y$};
  \path[->]
	(y) edge node[above] {$a$} (x)
	(x) edge node[above] {$b$} (ty);
   \end{tikzpicture}}}
\qquad \rightsquigarrow \qquad
    \vcenter{\hbox{\begin{tikzpicture}[point/.style={shape=circle,fill=black,scale=.5pt,outer sep=3pt},>=latex]
   \node[outer sep=-2pt] (x1) at (0,1) {$x$};
   \node[outer sep=-2pt] (y1) at (1,0) {$\red{\star^h}$};
   \node[outer sep=-2pt] (x2) at (2,1) {$\red{\star^t}$};
   \node[outer sep=-2pt] (y2) at (3,0) {$y$};
   \path[->]
	(x1) edge node[left] {$a$} (y1)
	(x2) edge node[right,pos=.3] {$b$} (y2);
    \draw[->,red,dashed] (x2) edge node[left,pos=.2] {$\red{\vec{\star}}$} (y1);
   \end{tikzpicture}}}
\end{equation}
Arrows of $Q$ retain their labels in $\tilde{Q}$, while a new arrow added from splitting a vertex $\star$ is denoted $\vec{\star}$, going from $\star^t$ to $\star^h$.
Letting $E\subset \tilde{Q}_1$ be the set of newly added arrows, we say that $Q$ is the \emph{contraction of $\tilde{Q}$ along $E$}.
The associated contraction map on vertices is denoted $\nu\colon \tilde{Q}_0 \to Q_0$.
A dimension vector $\bd$ for $Q$ determines a dimension vector $\tilde{\bd}$ for $\tilde{Q}$ by $\tilde{\bd}(z)=\bd(\nu(z))$.

When $Q^\epsilon$ is symmetric type $A$, it is straightforward to see that $\tilde{Q}$ has an even number of vertices, thus the (unique) involution $\tilde{\tau}$ for $\tilde{Q}$ fixes the middle arrow of $\tilde{Q}$.
To finish giving data that makes $\tilde{Q}$ a symmetric quiver, it is enough to specify that the value of $\tilde{s}$ on this arrow is also $\epsilon$.
We also note that we always have $G(\tilde{\bd})=GL(\tilde{\bd})$ since $\tilde{Q}$ has no $\tilde{\tau}$-fixed vertices.

\begin{proof}[Proof of Proposition \ref{prop:bipartitereduction}]
We first outline the general proof, then afterwards fill in the details which differ in the two cases.
In each case we define a closed embedding of algebraic groups \begin{equation}
\nu^*\colon G(\bd) \xrightarrow{\sim} H \leq G(\tilde{\bd})
\end{equation}
along with a closed embedding $\psi\colon\srep_Q(\bd) \hookrightarrow \srep_{\tilde{Q}}(\tilde{\bd})$
which is equivariant with respect to $\nu^*$.
Secondly, we show that the action of $G(\bd)$ on $\srep_Q(\bd)$ has a natural extension to an action of $G(\tilde{\bd})$ 
on $\srep_Q(\bd)$.  By Lemma \ref{lem:KXHG} this gives a $G(\tilde{\bd})$-equivariant isomorphism
\begin{equation}\label{eq:contractiso1}
 (G(\tilde{\bd})/H) \times \srep_Q(\bd)  
 \simeq G(\tilde{\bd}) *_{G(\bd)} \srep_Q(\bd)
\end{equation}
where $G(\tilde{\bd})$ acts diagonally on the left hand side and by left multiplication on the right hand side.
Finally, we will let $\mathscr{U} \subseteq \srep_{\tilde{Q}}(\tilde{\bd})$ be the open subvariety where the matrix over each contracted arrow is invertible,
and show that $\mathscr{U}$ is $GL(\tilde{\bd})$-equivariantly isomorphic to the left hand side of \eqref{eq:contractiso1}.
We now fill in the details of each case, with a diagram illustrating a small example in each.

\smallskip

\emph{Details when $Q$ has an even number of vertices.}
Since $G(\bd)=GL(\bd)$ in this case, the argument is very similar to \cite[\S2.5]{KR21}. A small example to orient the reader is in \eqref{eq:evenbipartite}.
\begin{equation}\label{eq:evenbipartite}
\vcenter{\hbox{\begin{tikzpicture}[point/.style={shape=circle,fill=black,scale=.5pt,outer sep=3pt},>=latex]
   \node (label) at (2.5,-1) {$Q$};
   \node[outer sep=-2pt] (y) at (0,0) {$y$};
  \node[outer sep=-2pt] (x) at (2,0) {$x$};
  \node[outer sep=-2pt] (tx) at (4,0) {$\tau(x)$};
  \node[outer sep=-2pt] (ty) at (6,0) {$\tau(y)$};
  \path[->]
	(y) edge node[above] {$a$} (x)
	(x) edge node[above] {$b=\tau(b)$} (tx)
	(tx) edge node[above] {$\tau(a)$} (ty);
\end{tikzpicture}}}
\qquad \rightsquigarrow \qquad
\vcenter{\hbox{\begin{tikzpicture}[point/.style={shape=circle,fill=black,scale=.5pt,outer sep=3pt},>=latex]
   \node (label) at (1.5,-1) {$\tilde{Q}$};
   \node[outer sep=-2pt] (z1) at (0,1) {$y$};
  \node[outer sep=-2pt] (z2) at (1,0) {$x^h$};
   \node[outer sep=-2pt] (z3) at (2,1) {$x^t$};
  \node[outer sep=-2pt] (tz3) at (3,0) {$\tilde{\tau}(x^t)$};
   \node[outer sep=-2pt] (tz2) at (4,1) {$\tilde{\tau}(x^h)$};
  \node[outer sep=-2pt] (tz1) at (5,0) {$\tilde{\tau}(y)$};
  \path[->]
	(z1) edge node[left] {$a$} (z2)
	(z3) edge node[left,pos=.2] {$\vec{x}$} (z2)
  	(z3) edge node[left,pos=.5] {$\substack{b=\\ \tilde{\tau}(b)}$} (tz3) 
	(tz2) edge node[left,pos=.2] {$\tilde{\tau}(\vec{x})$} (tz3)
	(tz2) edge node[right,pos=.3] {$\tau(a)$} (tz1);
   \end{tikzpicture}}}
\end{equation}

In this case we have the following decompositions of the representation variety and base change group:
\begin{equation}\label{eq:nodddecomps2}
    \srep_Q(\bd) = \Mat_{\epsilon}(\bd(hc), \bd(tc)) \times \prod_{a \in Q_1^+} \Mat(\bd(ha), \bd(ta)),
    \qquad 
    G(\bd) = \prod_{z \in Q_0^+} GL(\bd(z)),
\end{equation}
where $\epsilon$ is determined by the value of $s$ on the unique $\tau$-fixed arrow of $Q$.
Consider the subgroup 
\begin{equation}
    H:=\{(g_z) \in G(\tilde{\bd})\; \mid \; g_{ta}=g_{ha}\ 
    \text{if }a \in \tilde{Q}_1^+ \text{ is contracted}\}.
\end{equation}

Let $\nu^*\colon G(\bd) \xrightarrow{\sim} H$ be given by mapping $(g_w)_{w \in Q^+_0}$ to $(g_{\nu(z)})_{z \in \tilde{Q}^+_0}$.
Consider the closed embedding $\psi\colon\srep_Q(\bd) \hookrightarrow \srep_{\tilde{Q}}(\tilde{\bd})$
obtained by identifying $(V_a)_{a \in Q_1^+}$ with $(W_a)_{a \in \tilde{Q}_1^+}$ where $W_a = V_a$ if $a$ is not contracted, and $W_a$ is an appropriately sized identity matrix if $a$ is contracted.
This is equivariant with respect to $\nu^*$.

Extending the action of $G(\bd)$ on $\srep_Q(\bd)$ to an action of of $G(\tilde{\bd})$ in this case works exactly as in \cite[Lem.~2.17]{KR21}:
define a group homomorphism $\phi\colon G(\tilde{\bd}) \to G(\bd)$
by simply projecting away from the factors indexed by vertices of the form $z^h$ where $z\in Q_0$ has a length 2 path through it. Take the action of $G(\tilde{\bd})$ on $\srep_Q(\bd)$ obtained by extending along this group homomorphism.
Now a $GL(\tilde{\bd})$-equivariant isomorphism between $\mathscr{U}$ and the left hand side of \eqref{eq:contractiso1}
in this case is obtained exactly as in \cite[Prop.~2.18]{KR21}.

\smallskip

\emph{Details when $Q$ has an odd number of vertices.}
Let $\star \in Q_0$ be the middle vertex, the unique vertex fixed by $\tau$.  Exchanging the positive and negative parts if necessary, we can assume without loss of generality that the unique $c \in Q_1^+$ adjacent to $\star$ is oriented with $hc=\star$.
A small example to orient the reader is in \eqref{eq:splitex}, where $a$ in that example is $c$ here, $b=\tau(c)$, and so $\vec{\star}$ is fixed by $\tilde{\tau}$.
Let $\tilde{s}(\vec{\star}) = s(\star)$.

In this case we have the following decompositions of the representation variety and base change group:
\begin{equation}\label{eq:nodddecomps}
    \srep_Q(\bd) = \prod_{a \in Q_1^+} \Mat(\bd(ha), \bd(ta)),
    \qquad 
    G(\bd) = K_\star \times \prod_{z \in Q_0^+} GL(\bd(z))
\end{equation}
where $K:=K_\star$ is defined in \eqref{eq:Kz} with respect to a matrix $\Omega_\star$ which is symmetric if $s(\star)=+1$, or skew-symmetric if $s(\star)=-1$
(so $K\simeq O(\bd(\star))$ if $s(\star)=+1$ and $K\simeq Sp(\bd(\star))$ if $s(\star)=-1$).
Consider the subgroup 
\begin{equation}
    H:=\{(g_z) \in G(\tilde{\bd})\; \mid \; g_{ta}=g_{ha}\ 
    \text{if }a \in \tilde{Q}_1^+\setminus \{\vec{\star}\} \text{ is contracted},
    \text{ and }g_{\star^h} \in K \leq GL(\bd(\star^h))\}.
\end{equation}

Again $\nu^*\colon G(\bd) \xrightarrow{\sim} H$ is given by mapping $(g_w)_{w \in Q^+_0}$ to $(g_{\nu(z)})_{z \in \tilde{Q}^+_0}$.
Consider the closed embedding $\psi\colon\srep_Q(\bd) \hookrightarrow \srep_{\tilde{Q}}(\tilde{\bd})$
obtained by identifying $(V_a)_{a \in Q_1^+}$ with $(W_a)_{a \in \tilde{Q}_1^+}$ where $W_a = V_a$ if $a$ is not contracted, and $W_a$ is an appropriately sized identity matrix if $a$ is contracted.
This is equivariant with respect to $\nu^*$.

To extend the action of $G(\bd)$ on $\srep_Q(\bd)$ to an action of of $G(\tilde{\bd})$, we start as in the case of an even number of vertices by projecting from $G(\tilde{\bd})$ to $G' := GL(\bd(\star)) \times \prod_{z \in Q_0^+} GL(\bd(z))$, that is, away from factors indexed by vertices of the form $z^h$. 
With this, it is enough to extend the action of $G(\bd)$ to $G'$ then restrict along this homomorphism.  

To extend from $G(\bd)$ to $G'$, notice that the only difference is the factor indexed by $\star$, which is $K$ in $G(\bd)$ and $GL(\bd(\star))$ in $G'$.  But considering the definition of the action of $G(\bd)$ on $\srep_Q(\bd)$, we see that the factor $K$ is simply acting by left multiplication on the matrix space $\Mat(\bd(\star), \bd(tc))$,
and thus canonically extends to an action of $GL(\bd(\star))$ by left multiplication.
As in \eqref{eqn:MPembedding}, the open subvariety of invertible matrices in $\Mat_\epsilon(\bd(\star),\bd(\star))$ is isomorphic to $GL(\bd(\star))/K$.
Then showing that $\mathscr{U}$ is $GL(\tilde{\bd})$-equivariantly isomorphic to the left hand side of \eqref{eq:contractiso1}
can be carried out by modifying the proof of \cite[Prop.~2.18]{KR21} using this identification.
\end{proof}

We note the following corollary for reference later, which can be immediately extracted from the preceding proof.

\begin{corollary}\label{cor:quotsmoothaffine}
The quotient $G(\tilde{\bd})/G(\bd)$ is smooth and affine.
\end{corollary}

\section{Proof of the main theorem}\label{sec:mainThmProof}
In this section we prove Theorem \ref{thm:mainTheorem}, and fix the following notation throughout.
Given a type $A$ symmetric quiver $Q^\epsilon$ with symmetric dimension vector $\bd$, let $(\tilde{Q^\epsilon}, \tilde{\bd})$ be the associated symmetric bipartite type $A$ quiver and dimension vector as in Section \ref{sec:bipartitereduction}.  Recall that $\tilde{Q}$ has an even number of vertices and middle arrow $\vec{\star}$ fixed by $\tilde{\tau}$.
Denote the total dimension of $\tilde{\bd}$ by $\tilde{d}$, which is even, and let $G=GL(\tilde{d})$ and $K=G^\theta$ as defined in \eqref{eq:theta}, using the specific form $\Omega_{\epsilon}$ defined in \eqref{eq:omegapm}.
Let $P \leq G(\tilde{\bd})$ be the parabolic subgroup of block lower-triangular matrices
with block sizes $\bd(y_0), \bd(y_1),\dots, \bd(y_{n-1}), \bd(x_n),\dots, \bd(x_1)$ down the diagonal.
Recall from \eqref{eq:zetaKdef} the symmetric Zelevinsky map
\[
\zeta^\epsilon\colon \srep_{\tilde{Q}}(\tilde{\bd})\rightarrow Y^\epsilon_d.
\]
The maximal tori $T(\tilde{\bd}) \leq G(\tilde{\bd})$ and $T(\bd) \leq G(\bd)$ consist of matrices which are diagonal in each factor.

We prove one technical lemma first which is needed to understand the maps on equivariant $K$-groups.

\begin{lemma}\label{lem:CMthing}
Let $Y \hookrightarrow X$ be a regular embedding of schemes, and $Z \subset X$ a Cohen-Macaulay closed subscheme such that 
the codimension of $Z$ in $X$ is equal to the codimension of $Z \cap Y$ in $Y$.
Then $\Tor_i^{\mathcal{O}_X}(\mathcal{O}_Y, \mathcal{O}_Z)=0$ for all $i \geq 1$.
\end{lemma}
\begin{proof}
We reduce to the local case where $X=\Spec(S)$, $Y=\Spec(S/J)$, and $Z=\Spec(S/I)$ for some ideals $I, J \subset (S,\mathfrak{m})$. The hypotheses on codimension mean that 
\begin{equation}
\dim S - \dim S/I = \dim S/J - \dim S/(I+J).
\end{equation}
The hypothesis that $Y \hookrightarrow X$ is a regular embedding means that $J$ is generated by a regular sequence $(s_1, ..., s_r)$, and so $\dim S/J=\dim S - r$.
Plugging in above and simplifying yields:
\begin{equation}
\dim S/(I+J) = \dim S/I - r.
\end{equation}
Now \cite[Thm.~2.1.2(c)]{BH93} with $M=S/I$ gives that $(s_1, ..., s_r)$ is still a regular sequence on $S/I$. 
Finally, \cite[Exer.~1.1.12(b)]{BH93} with $M=S/I$ and ${\bf x} = (s_1, ..., s_r)$ implies that $\Tor^S_i(S/I, S/J)=0$ for all $i \geq 1$. 
\end{proof}

\begin{proof}[Proof of Theorem \ref{thm:mainTheorem}]
(i)  Let $\mathscr{O}^{\ddagger}:=\overline{P\cdot \zeta^\epsilon(1*\mathscr{O})}$ in $\text{Mat}^\circ_\epsilon(2d,2d)$ and recall the isomorphism $j_\epsilon$ from \eqref{eqn: IsoMatrixSpaceSymmSpace}. Then the map of \eqref{eq:reptoKG} is defined by $\mathscr{O}^\dagger:=j_\epsilon(\mathscr{O}^{\ddagger})$. Since $j_\epsilon$ is a $G$-equivariant isomorphism, and thus a $P$-equivariant isomorphism, it suffices to prove that the map which sends $\mathscr{O}$ to $\mathscr{O}^\ddagger$ is an injective, order preserving map of partially ordered sets:
 \begin{equation}\label{eq:reptoMat}
\begin{split}
\left\{\begin{tabular}{c} $G(\bd)$-orbit closures\\ in $\srep_{Q}^\epsilon(\bd)$\end{tabular} \right\}
& \rightarrow
\left\{\begin{tabular}{c} $P$-orbit closures\\ in $\text{Mat}^\circ_\epsilon(2d,2d)$ \end{tabular} \right\}.
\end{split}
\end{equation}

We begin by noting that $\mathscr{O}^\ddagger$ is indeed a $P$-orbit closure in $\text{Mat}^\circ_\epsilon(2d,2d)$ 
since it is an irreducible, $P$-stable, closed subvariety of $\text{Mat}^\circ_\epsilon(2d,2d)$,
which has finitely many $P$-orbits. 

To see that the map \eqref{eq:reptoMat} 
is injective, we first show that $\mathscr{O}^\circ \mapsto P\cdot \zeta^\epsilon(1*\mathscr{O}^\circ)$ is an injective map between the associated sets of orbits. Since each orbit closure is determined by its open orbit, this is enough. 
Choose points $V_i \in \mathscr{O}_i^\circ$ ($i=1,2$), and suppose that $\zeta^\epsilon(1*V_1)$ and $\zeta^\epsilon(1*V_2)$ are in the same $P$-orbit.
By Lemma \ref{lem:PtoGdorbits}, we have that $1*V_1$ and $1*V_2$ are in the same $G(\tilde{\bd})$-orbit in $\srep_{\tilde{Q}}(\tilde{\bd})$, which implies that $V_1$ and $V_2$ lie in the same $G(\bd)$-orbit by Proposition \ref{prop:bipartitereduction}. 

To justify that \eqref{eq:reptoMat} is order preserving, suppose that $\mathscr{O}_1\subseteq \mathscr{O}_2$ in $\srep_{Q}(\bd)$. Then $1*\mathscr{O}_1\subseteq 1*\mathscr{O}_2$, and hence $\overline{P\cdot \zeta^\epsilon(1*\mathscr{O}_1)}\subseteq \overline{P\cdot \zeta^\epsilon(1*\mathscr{O}_2)}.$

(ii) It suffices to show that there is a smooth affine variety $X$ and a smooth morphism  
\begin{equation}
\varphi: X \times \srep_{Q}^\epsilon(\bd)\rightarrow \srep_{Q}^\epsilon(\bd)^\ddagger\subset \text{Mat}^\circ_\epsilon(2d,2d)    
\end{equation}
which restricts to a smooth morphism 
\begin{equation}\label{eq:maintheoremmap}
        X \times \mathscr{O} \rightarrow \mathscr{O}^\ddagger.
\end{equation}

Let $Y_d^\epsilon\subseteq \Mat_\epsilon^\circ(2d,2d)$ be the matrix space from \eqref{eqn:MatrixSlice}. By Lemma \ref{lem:MSSIsomorphism} we get a smooth map
\begin{equation}\label{eqn:smoothmap}
    B_-\times (\mathscr{O}^\ddagger \cap Y_d^\epsilon)\rightarrow \mathscr{O}^\ddagger, \quad (b,y)\mapsto byb^t.
\end{equation}
Similar to \eqref{eq:piretract}, we have a retraction of $\zeta^\epsilon$ given by extracting certain submatrices in \eqref{eq:zetaKdef} and multiplying certain entries by $\epsilon$:
\begin{equation}\label{eq:piepsilon}
\pi^\epsilon\colon Y_d^\epsilon \to \srep_{\tilde{Q}}(\tilde{\bd}) \qquad \text{satisfying} \quad \pi^\epsilon\circ \zeta^\epsilon = 1_{\srep_{\tilde{Q}}(\tilde{\bd})}.
\end{equation}
This makes $Y_d^\epsilon$ the total space of a vector bundle on $\srep_{\tilde{Q}}(\tilde{\bd})$.
Denote the pullback to $Y_d^\epsilon$ of the open subvariety $\mathscr{U} \subseteq \srep_{\tilde{Q}}(\tilde{\bd})$ of Proposition \ref{prop:bipartitereduction} by $Y^\circ$, which is therefore open in $Y_d^\epsilon$.

Next we want to show that $\mathscr{O}^\ddagger \cap Y^\circ = \zeta^\epsilon(G(\tilde{\bd})\cdot\mathscr{O})$ where the intersection is scheme-theoretical. 
We will do this by showing they are both 
reduced and irreducible closed subschemes of $Y^\circ$ which have a dense subset of $\kk$-points in common (here we use that $\kk$ is algebraically closed).

First observe that $\mathscr{O}^\ddagger \cap Y^\circ$ is a reduced and irreducible scheme by Lemma \ref{lem:orbitcapSirred}, since it is open in $\mathscr{O}^\ddagger \cap Y_d^\epsilon$.
Furthermore, since $\mathscr{O}^\ddagger$ is a closed subscheme of $\Mat_\epsilon^\circ(2d,2d)$, the intersection $\mathscr{O}^\ddagger \cap Y^\circ$ is a closed subscheme of $Y^\circ$.
Now $\zeta^\epsilon(G(\tilde{\bd})\cdot\mathscr{O})$ is a reduced and irreducible closed subscheme of $Y^\circ$,
since $\zeta^\epsilon$ is a closed embedding with $\zeta^\epsilon(\mathscr{U}) \subseteq Y^\circ$,
and $G(\tilde{\bd})\cdot\mathscr{O}=G(\tilde{\bd})*_{G(\bd)}\mathscr{O}$ is a reduced and irreducible closed subscheme of $\mathscr{U}$.

By Proposition \ref{prop:POrbitimage},
we have equalities
$(\mathscr{O}^\ddagger)^\circ \cap Y^\circ =
(\mathscr{O}^\ddagger)^\circ \cap Y_d^\epsilon = \zeta^\epsilon(G(\tilde{\bd})\cdot \mathscr{O}^\circ)$ at the level of $\kk$-points and thus $\mathscr{O}^\ddagger \cap Y^\circ = \zeta^\epsilon(G(\tilde{\bd})\cdot\mathscr{O})$ holds as schemes.

Finally, the fact that $\zeta^\epsilon$ is a closed embedding gives the middle isomorphism below, while Proposition \ref{prop:bipartitereduction} together with Lemma \ref{lem:KXHG} gives the last isomorphism of
\begin{equation}\label{eq:mainlastisos}
\mathscr{O}^\ddagger \cap Y^\circ 
=\zeta^\epsilon(G(\tilde{\bd})\cdot\mathscr{O})
\simeq G(\tilde{\bd}) \cdot \mathscr{O} 
\simeq (G(\tilde{\bd})/G(\bd)) \times \mathscr{O}.
\end{equation}

Now the smooth map \eqref{eq:maintheoremmap} of the main theorem is obtained as the composition:
\begin{equation}
    B_-\times (G(\tilde{\bd})/G(\bd)) \times \mathscr{O}
    \simeq
     B_-\times (\mathscr{O}^\ddagger \cap Y^\circ) \hookrightarrow
     B_-\times (\mathscr{O}^\ddagger \cap Y_d^\epsilon)\rightarrow \mathscr{O}^\ddagger
\end{equation}
where the first isomorphism uses \eqref{eq:mainlastisos}, the second map is an open embedding (so smooth), and the final map is the smooth map \eqref{eqn:smoothmap}.
So, explicitly, $X=
 B_-\times (G(\tilde{\bd})/G(\bd))$ in the notation of \eqref{eq:maintheoremmap}.
Note that the quotient $G(\tilde{\bd})/G(\bd)$ is smooth and affine by Corollary \ref{cor:quotsmoothaffine}.

(iii) 
Once again, we use the $T$-equivariant isomorphism $j_\epsilon$ to work in the space $\text{Mat}^\circ_\epsilon(2d,2d)$ instead of $G/K$. That is, we will show
that there is a homomorphism of equivariant Grothendieck groups
\begin{equation}\label{eq:homGrothGroupsMat}
{K}_T(\text{Mat}^\circ_\epsilon(2d,2d)) \to {K}_{T(\bd)}(\srep_{Q}^\epsilon(\bd))
\end{equation}
sending the class of the structure sheaf $[\mathcal{O}_{\mathscr{O}^\ddagger}]$ to $[\mathcal{O}_\mathscr{O}]$ whenever $\mathscr{O}^\ddagger$ is Cohen-Macaulay.

Let $T_Y=\bar{\xi}(T(\tilde{\bd})) \leq GL_{2d}$ where $\bar{\xi}$ is the embedding of \eqref{eq:GdtoGL}.
Restricting the group action we get a homomorphism 
${K}_T(\text{Mat}^\circ_\epsilon(2d,2d))\to {K}_{T_Y}(\text{Mat}^\circ_\epsilon(2d,2d))$.
Now the map 
\[
\Psi\colon
B_-\times Y^\epsilon_d \rightarrow \text{Mat}^\circ_\epsilon(2d,2d), \quad (b,y)\mapsto byb^t
\] 
is $T_Y$-equivariant 
with respect to $T_Y$ acting on $B_-\times Y^\epsilon_d$ by $h\cdot(b,y) = (hbh^{-1},hyh^t)$ and $T_Y$ acting on $\text{Mat}^\circ_\epsilon(2d,2d)$ by $h\cdot M = hMh^t$. 
Also $\Psi$ is smooth by Lemma \ref{lem:MSSIsomorphism}.
Pullback of sheaves along $\Psi$
induces
\begin{equation}
    {K}_{T_Y}(\text{Mat}^\circ_\epsilon(2d,2d)) \to {K}_{T_Y}(
    B_-\times Y^\epsilon_d) \quad \text{such that}\quad [\mathscr{F}]\mapsto [\Psi^*\mathscr{F}],
\end{equation}
the class of a sheaf is sent to the class of its pullback because smooth maps are flat and thus pullback of coherent sheaves along a smooth map is an exact functor \cite[5.2.5]{CGbook}.
In particular, we have $[\mathcal{O}_{\mathscr{O}^\ddagger}]\mapsto [\Psi^* \mathcal{O}_{\mathscr{O}^\ddagger}] =[\mathcal{O}_{
 B_-\times (\mathscr{O}^\ddagger \cap Y^\epsilon_d)}]$.
Since
 $B_-$ is smooth, we have a $T_Y$-equivariant regular embedding and induced restriction map \cite[\S5.2.5]{CGbook}:
\begin{equation}\label{eq:StoBS}
Y^\epsilon_d \hookrightarrow 
 B_-\times Y^\epsilon_d, \quad y\mapsto (1,y), \quad {K}_{T_Y}(
 B_-\times Y^\epsilon_d) \to {K}_{T_Y}(Y^\epsilon_d).
\end{equation}
When $\mathscr{O}^\ddagger$ is Cohen-Macaulay, $[\mathcal{O}_{
B_-\times (\mathscr{O}^\ddagger \cap Y^\epsilon_d)}]$ gets mapped to $[\mathcal{O}_{\mathscr{O}^\ddagger \cap Y^\epsilon_d}]$ since the higher Tor sheaves vanish by Lemma \ref{lem:CMthing}.
See \cite[Rmk.~5.2.7]{CGbook} and the discussion directly above it for more detail about this step.

Noting that $T_Y \xrightarrow{\sim} T(\tilde{\bd})$,
it is straightforward to see that $\pi^\epsilon$ in \eqref{eq:piepsilon} is equivariant with respect to this isomorphism.  
The restriction of $\pi^\epsilon$ to the $T_Y$-stable open subvariety $Y^\circ$ is thus an equivariant vector bundle over the open subvariety $\mathscr{U}$ of Proposition \ref{prop:bipartitereduction}, inducing maps
\begin{equation}
{K}_{T_Y}(Y_d^\epsilon) \twoheadrightarrow {K}_{T_Y}(Y^\circ) \xrightarrow{\sim} {K}_{T(\tilde{\bd})}(\mathscr{U})
\end{equation} 
sending the class of a sheaf to the class of its restriction along the equivariant section $\zeta^\epsilon|_{\mathscr{U}}$ (see \cite[Cor.~5.4.21]{CGbook} or \cite[Thm.~1.7]{Thomason88}).
In particular, the class at the end of the previous paragraph $[\mathcal{O}_{\mathscr{O}^\ddagger \cap Y_d^\epsilon}] \in {K}_{T_Y}(Y_d^\epsilon)$ is sent to $[\mathcal{O}_{G(\tilde{\bd})\cdot\mathscr{O}}] \in {K}_{T(\tilde{\bd})}(\mathscr{U})$. 

We obtain homomorphisms ${K}_{T(\tilde{\bd})}(\mathscr{U}) \to {K}_{T(\bd)}(\mathscr{U})$ by restricting the torus action via the identification $T(\bd) \leq T(\tilde{\bd})$, and ${K}_{T(\bd)}(\mathscr{U}) \to {K}_{T(\bd)}(\rep_Q(\bd))$ by pulling back classes along the $T(\bd)$-equivariant regular embedding $\rep_Q(\bd) \hookrightarrow \mathscr{U}$ (see Proposition \ref{prop:bipartitereduction}).
Composing these two maps with the sequence of maps from the above paragraphs completes the definition of the desired homomorphism \eqref{eq:homGrothGroupsMat}.  
The fact that the last two steps send $[\mathcal{O}_{GL(\tilde{\bd})\cdot\mathscr{O}}] \mapsto [\mathcal{O}_{\mathscr{O}}]$ 
when $\mathscr{O}^\ddagger$ is Cohen-Macaulay uses Lemma \ref{lem:CMthing} again,
where the hypothesis that  $GL(\tilde{\bd})\cdot\mathscr{O} \simeq \mathscr{O}^\ddagger \cap Y^\circ$ is Cohen-Macaulay is satisfied by Lemma \ref{lem:MSSIsomorphism}.

Part (iv) follows by tracing through the analogous maps on Chow groups in the same way as part (iii).  The difference is that the Cohen-Macaulay condition is not needed to keep track of the class at the points where this condition is invoked in the $K$-theory proof.
\end{proof}

\begin{remark}
The bipartite Zelevinsky map from \cite{KR15} (see Section \ref{sec:ordinaryZmap}) induces an isomorphism from each bipartite type $A$ quiver orbit closure to a Kazhdan-Lusztig variety in a type $A$ flag variety. As prime defining ideals of Kazhdan-Lusztig varieties are known (see \cite{WooYong08}), one easily obtains prime defining ideals of bipartite type $A$ quiver orbit closures.

In contrast, while we obtain a similar geometric isomorphism,
our methods do not yield generators for the prime defining ideals. It would be interesting to find such equations in the type $A$ symmetric quiver setting.
\end{remark}

\section{Converse theorems}\label{sec:converse}

In this section we prove some straightforward results which can loosely be viewed as converses to our main theorem, analogous to \cite[Thm.~2.6]{KR21} for type $D$ quivers and \cite[Rmk.~2.13]{KR21} for ordinary type $A$ quivers.
Since the proofs are quite similar to the type $D$ case and are not our main results, we refer to reader to \cite[\S2]{KR21} for more background and detail.
We retain the conventions of \S\ref{subsec:Conventions} for clarity and conciseness, in particular $G=GL(2d)$.  However, we remark that the theorems of this section should naturally generalize to the case of odd rank in the orthogonal case with the same proofs.

\begin{theorem}\label{thm:converse1}
Consider the flag variety $G/B_-$ and symmetric subgroup $K=O(2d)$ or $K=Sp(2d)$ of $G$ as defined in \S\ref{subsec:Conventions}. There exists a symmetric quiver of Dynkin type $A$ with symmetric dimension vector $\bd$, and a $G(\bd)$-stable open subvariety $U \subset \srep_Q(\bd)$ such that the following hold.
\begin{enumerate}[(i)]
    \item There is an injective, order preserving map of partially ordered sets
 \begin{equation}
\begin{split}
\left\{\begin{tabular}{c} $K$-orbit closures\\ in $GL(2d)/B$\end{tabular} \right\}
& \to
\left\{\begin{tabular}{c} $G(\bd)$-orbit closures\\ in $\srep_Q(\bd)$ \end{tabular} \right\}
\end{split}
\end{equation}
    which we denote by $\mathscr{O} \rightarrow \mathscr{O}'$ below. The image of the map is the set of orbit closures which have nontrivial intersection with $U$.
    \item Any smooth equivalence class of singularity occurring in $\mathscr{O}$ also occurs in $\mathscr{O}'$.
    \item There is a homomorphism of equivariant Grothendieck groups 
    \[
    f\colon K_{G(\bd)}(\srep_Q(\bd))\to K_T(G/B)
    \]
    where $T$ is a maximal torus of $K$, and such that $[\mathcal{O}_{\mathscr{O}}]=f([\mathcal{O}_{\mathscr{O}'}])$,
    where $\mathcal{O}$ denotes the structure sheaf in each case.
\end{enumerate}
\end{theorem}
\begin{proof}
Consider the type $A_{4d-1}$ quiver and dimension vector
\begin{equation}
\vcenter{\hbox{\begin{tikzpicture}[point/.style={shape=circle,fill=black,scale=.5pt,outer sep=3pt},>=latex]
  \node[outer sep=-2pt] (5') at (-6,0) {$1$};
   \node[outer sep=-2pt] (4') at (-5,0) {$2$};
   \node[outer sep=-2pt] (4'a) at (-4,0) {};
   \node[outer sep=-2pt] (d') at (-3.5,0) {${\cdots}$};
   \node[outer sep=-2pt] (3'a) at (-3,0) {};
  \node[outer sep=-2pt] (2') at (-2,0) {$2d-1$};
   \node[outer sep=-2pt] (1) at (-0.5,0) {$2d$};
  \node[outer sep=-2pt] (3) at (1,0) {$2d-1$};
   \node[outer sep=-2pt] (3a) at (2,0) {};
   \node[outer sep=-2pt] (d) at (2.5,0) {${\cdots}$};
   \node[outer sep=-2pt] (4a) at (3,0) {};
   \node[outer sep=-2pt] (4) at (4,0) {$2$};
  \node[outer sep=-2pt] (5) at (5,0) {$1$};
  \path[->]
    (5') edge node[auto] {} (4')
    (4') edge node[auto] {} (4'a)
    (3'a) edge node[auto] {} (2')
  	(2') edge node[auto] {} (1)
	(1) edge node[auto] {} (3)
	(3) edge node[auto] {} (3a)
	(4a) edge node[auto] {} (4)
	(4) edge node[auto] {} (5); 
	\end{tikzpicture}}}.
\end{equation}
Let $z$ be the central vertex, where $\bd(z)=2d$.  
The involution $\tau$ on $Q$ is uniquely determined and we can take
\begin{equation}
    \begin{cases}
    s(z) = 1 & \text{if }K=O(2d)\\
    s(z) = -1 & \text{if }K=Sp(2d).\\
    \end{cases}
\end{equation}
Here we have
\begin{equation}
    \srep_Q(\bd) = \prod_{i=1}^{2d-1} \Mat(i,i+1).
\end{equation}
Consider the open subvariety $U \subset \srep_Q(\bd)$ where all matrices have maximal rank, and decompose $G(\bd) = G' \times K$ where $G' = \prod_{i=1}^{2d} GL(i)$.
Taking the quotient of the action of $G'$ on $U$, we get the variety of complete flags in $\kk^{2d}$, so $U/G'$ is identified with $G/B$.
This isomorphism is $K$-equivariant with respect to the unused $K$-action on $U/G'\simeq G/B$.

Statements (i) and (ii) then follow from standard properties of geometric quotients,
and for part (iii) we take the composition of standard group homomorphisms in equivariant $K$-theory
\[
f\colon {K}_{G(\bd)}(\srep_Q(\bd)) \twoheadrightarrow {K}_{G(\bd)}(U) \simeq {K}_K(U/G') \simeq {K}_K(G/B) \hookrightarrow {K}_T(G/B). \qedhere
\]
\end{proof}

\begin{remark}
Unlike in item (iii) of the Main Theorem (Theorem \ref{thm:mainTheorem}), we do not need to assume that $\mathscr{O}$ is Cohen-Macaulay in item (iii) of Theorem \ref{thm:converse1}, or in item (iii) of Theorem \ref{thm:converse2} below. 
\end{remark}

The proof of the next converse theorem is only slightly more complicated because we need to go through a homogeneous fiber bundle in this case.

\begin{theorem}\label{thm:converse2}
Given a space of (skew-)symmetric matrices $\Mat_\epsilon(2d)$, there exists a symmetric quiver of Dynkin type $A$ with symmetric dimension vector $\bd$, and a $G(\bd)$-stable open subvariety $U \subset \srep_Q(\bd)$ such that the following hold.
\begin{enumerate}[(i)]
    \item There is an injective, order preserving map of partially ordered sets
 \begin{equation}
\begin{split}
\left\{\begin{tabular}{c} $B$-orbit closures\\ in $\Mat_\epsilon(2d)$\end{tabular} \right\}
& \to
\left\{\begin{tabular}{c} $G(\bd)$-orbit closures\\ in $\srep_Q(\bd)$ \end{tabular} \right\}
\end{split}
\end{equation}
    which we denote by $\mathscr{O} \rightarrow \mathscr{O}'$ below. The image of the map is the set of orbit closures which have nontrivial intersection with $U$.
    \item Any smooth equivalence class of singularity occurring in $\mathscr{O}$ also occurs in $\mathscr{O}'$.
    \item There is a homomorphism of equivariant Grothendieck groups 
    \[
    f\colon {K}_{G(\bd)}(\srep_Q(\bd))\to {K}_T(\Mat_{\epsilon}(2d))
    \]
    where $T$ is the maximal torus of diagonal matrices in $GL(2d)$, 
    and such that 
    $[\mathcal{O}_{\mathscr{O}}]=f([\mathcal{O}_{\mathscr{O}'}])$,
    where $\mathcal{O}$ denotes the structure sheaf in each case.
\end{enumerate}
\end{theorem}
\begin{proof}
Consider the type $A_{4d}$ quiver and dimension vector
\begin{equation}
\vcenter{\hbox{\begin{tikzpicture}[point/.style={shape=circle,fill=black,scale=.5pt,outer sep=3pt},>=latex]
  \node[outer sep=-2pt] (5') at (-6,0) {$1$};
   \node[outer sep=-2pt] (4') at (-5,0) {$2$};
   \node[outer sep=-2pt] (4'a) at (-4,0) {};
   \node[outer sep=-2pt] (d') at (-3.5,0) {${\cdots}$};
   \node[outer sep=-2pt] (3'a) at (-3,0) {};
  \node[outer sep=-2pt] (2') at (-2,0) {$2d-1$};
   \node[outer sep=-2pt] (1') at (-0.5,0) {$2d$};
   \node[outer sep=-2pt] (1) at (0.5,0) {$2d$};
  \node[outer sep=-2pt] (3) at (2,0) {$2d-1$};
   \node[outer sep=-2pt] (3a) at (3,0) {};
   \node[outer sep=-2pt] (d) at (3.5,0) {${\cdots}$};
   \node[outer sep=-2pt] (4a) at (4,0) {};
   \node[outer sep=-2pt] (4) at (5,0) {$2$};
  \node[outer sep=-2pt] (5) at (6,0) {$1$};
  \path[->]
    (5') edge node[auto] {} (4')
    (4') edge node[auto] {} (4'a)
    (3'a) edge node[auto] {} (2')
  	(2') edge node[auto] {} (1')
	(1') edge node[auto] {$a$} (1)
	(1) edge node[auto] {} (3)
	(3) edge node[auto] {} (3a)
	(4a) edge node[auto] {} (4)
	(4) edge node[auto] {} (5); 
	\end{tikzpicture}}},
\end{equation}
where we have only labeled the middle arrow, call it $a$.
The involution $\tau$ on $Q$ is uniquely determined and we can take $s(a)=\epsilon$, giving
\begin{equation}\label{eq:conv1a}
    \srep_Q(\bd) = \prod_{i=1}^{2d-1} \Mat(i,i+1) \times \Mat_{\epsilon}(2d).
\end{equation}
Consider the open subvariety $U \subset \srep_Q(\bd)$ where all matrices besides the one over $a$ have maximal rank, and decompose $G(\bd) = G' \times G$ where $G' = \prod_{i=1}^{2d} GL(i)$ and $G=GL(2d)$.  Taking the quotient of the action of $G'$ on $U$, the first $2d-1$ factors of \eqref{eq:conv1a} become the variety of complete flags in $\kk^{2d}$, which we identify with $GL(2d)/B$ as usual, so $U/G'$ is identified with $G/B \times \Mat_{\epsilon}(2d)$.
This isomorphism is $G$-equivariant with respect to the unused $G$-action on $U/G'$, which translates to the diagonal action $g \cdot (xB, A)=(gxB, gAg^T)$.
Putting this together we then get a $G$-equivariant isomorphism with a homogeneous fiber bundle
\begin{equation}
    U/G' \simeq G/B \times \Mat_{\epsilon}(2d) \simeq G *_B \Mat_{\epsilon}(2d).
\end{equation}
where the second isomorphism uses Lemma \ref{lem:KXHG}.

Statements (i) and (ii) then follow from standard properties of homogeneous fiber bundles and geometric quotients,
and for part (iii) we take the composition of standard group homomorphisms in equivariant $K$-theory
\[
f\colon K_{G(\bd)}(\srep_Q(\bd)) \twoheadrightarrow K_{G(\bd)}(U) \simeq K_G(U/G') \simeq K_G(G *_B \Mat_{\epsilon}(n)) \simeq K_T(\Mat_{\epsilon}(2d)).
\]
See the proof of \cite[Thm.~2.6]{KR21} for more detail.
\end{proof}

We would like to know how some compactifications might fit into the story of this paper, analogous to the way that double Grassmannians appear in the type $D$ quiver setting.
\begin{question}\label{ques:compact}
Can compactifications of $GL(2d)/Sp(2d)$ and $GL(2d)/O(2d)$ (e.g. varieties of complete quadrics \cite{Semple48,Laksov85}) be worked into the story of this paper?
\end{question}

\section{Singularities of orbit closures in \texorpdfstring{\\}{} type \texorpdfstring{$A$}{Lg} symmetric quiver representation varieties}\label{sec:singularities}

In this short section, we study the singularities of symmetric type $A$ quiver orbit closures. For convenience, we begin by summarizing item (ii) of our main theorem (Theorem \ref{thm:mainTheorem}) and item (ii) of our converse  theorem (Theorem \ref{thm:converse1}) into one statement. This is a \emph{symmetric} type $A$ quiver analogue of a result of Bobi\'{n}ski and Zwara on singularities of orbit closures in representation varieties of (ordinary) type $A$ quivers (see \cite{BZ02}).

Fix $\epsilon = \pm 1$. Following similar notation to \cite{BZ02}, let $\text{Sing}(\mathbb{A}_\epsilon)$ denote the set of all smooth equivalence classes of singularities that occur in  $G(\bd)$-orbit closures in type $A$ symmetric representation varieties $\text{srep}_{Q}(\bd)$ (with the fixed value of $\epsilon$). 
Let $\text{Sing}(GL/O)$ (respectively $\text{Sing}(GL/Sp)$) denote the set of all smooth equivalence classes of singularities that occur in $B_-$ orbit closures in symmetric varieties $GL(2d)/O(2d)$ (respectively $GL(2d)/Sp(2d)$).

\begin{theorem}\label{thm:sing}
 $\text{Sing}(\mathbb{A}_+) = \text{Sing}(GL/O)$ and $\text{Sing}(\mathbb{A}_-) = \text{Sing}(GL/Sp)$.
\end{theorem}

\begin{proof}
This is immediate from item (ii) of Theorem \ref{thm:mainTheorem} together with item (ii) Theorem \ref{thm:converse1}.
\end{proof}

Combining work of Brion \cite{Brion01, Brion03} with  Theorem \ref{thm:sing} yields the following corollary.

\begin{corollary}\label{cor:sing}
Let $Q$ be a symmetric quiver of Dynkin type $A$ with $\epsilon = -1$, and let $\bd$ be a symmetric dimension vector for $Q$. Then each $G(\bd)$-orbit closure in $\text{srep}_Q(\bd)$ is normal and Cohen-Macaulay. Under the additional assumption that the characteristic of the base field $\kk$ is $0$, each $G(\bd)$-orbit closure in $\text{srep}_Q(\bd)$ has rational singularities.
\end{corollary}

\begin{proof}
Notice that the singularities that occur in $B_-$ orbit closures in $GL(2d)/Sp(2d)$ are the same as the singularities that occur in $Sp(2d)$ orbit closures in $GL(2d)/B_-$. 

By \cite[Corollary 2]{Brion01}, $Sp(2d)$-orbit closures in $GL(2d)/B_-$ are multiplicity free. We remark that \cite{Brion01} is about complex algebraic groups, but the characterization of multiplicity free spaces is characteristic independent as it only depends on the poset of $B_-$-orbits on $GL(2d)/Sp(2d)$ which does not depend on the characteristic of $\kk$.  Thus  we can apply \cite[Theorem 2]{Brion03} and deduce normality and Cohen-Macaulayness of  $Sp(2d)$-orbit closures in the flag variety from normality and Cohen-Macaulayness of the flag variety $GL(2d)/B_-$. Moreover, if the characteristic of the base field is zero, by \cite[Remark 3]{Brion03} a multiplicity free subvariety of the flag variety has rational singularities and by applying once again \cite[Corollary 2]{Brion01} we conclude that any $Sp(2d)$ orbit closure in $GL(2d)/B_-$ has rational singularities. Now the claim follows immediately from Theorem \ref{thm:sing}. 
\end{proof}

The situation when $\epsilon = +1$
is quite different as there are Borel orbit closures in $GL(2d)/O(2d)$ which are neither normal nor Cohen-Macaulay (see \cite[\S7.2.1 and \S7.2.2]{Pin01}), as the next example shows:

\begin{example}\label{ex:nonCM} Let $x\in GL(6)$ be the following matrix:
\[
x=
\left[
\begin{array}{cccccc}
     1& 0&0&0&0&0\\
     0 &1&0&0&0&0\\
      0 &0&0&0&0&1\\
    0    &0&0&1&0&0\\
        0 &0&0&0&1&0\\
         0 &0&1&0&0&0
\end{array}
\right].
\]
By applying $\iota_+$ to its coset $xK$ we obtain the symmetric matrix 
\[
w=x\Omega_+x^t=\left[
\begin{array}{cccccc}
      0&0&1&0&0&0 \\
      0&0&0&0&1&0 \\
      1&0&0&0&0&0 \\
      0&0&0&0&0&1\\
      0&1&0&0&0&0\\
      0&0&0&1&0&0
\end{array}
\right]
\]
which corresponds to the involution $351624\in \mathfrak{S}_6$ written in one-line notation.

Let $G = GL(6)$ and $K = O(6)$. Then, according to Proposition \ref{prop:matrixNWranks}, the variety  $\overline{Bw}\subseteq \Mat_+^\circ(6,6)$ has equations given by the northwest rank conditions contained in the following matrix:
\[
\left[
\begin{array}{cccccc}
     0&0&1&1&1&1\\
     0&0&1&1&2&2\\
     1&1&2&2&3&3\\
     1&1&2&2&3&4\\
     1&2&3&3&4&5\\
     1&2&3&4&5&6
\end{array}
\right].
\]
This is a small enough example that the failure of normality and Cohen-Macaulayness can be checked in Macaulay2 \cite{M2}.  On the other hand, 
consider the symmetric quiver representation 
\begin{equation}\label{eq:specialsymmetricquiverrep}
\kk^1\xrightarrow{I_{2,1}} \kk^2\xrightarrow{I_{3,2}} \kk^3\xrightarrow{I_{4,3}} \kk^4\xrightarrow{I_{5,4}} \kk^5\xrightarrow{I_{6,5}} \kk^6\xrightarrow{w} \kk^6\xrightarrow{I_{5,6}} \kk^5\xrightarrow{I_{4,5}} \kk^4\xrightarrow{I_{3,4}} \kk^3\xrightarrow{I_{2,3}} \kk^2\xrightarrow{I_{1,2}} \kk^1
\end{equation}
where $I_{a,b}$ denotes an $a\times b$ matrix with $1$s down the diagonal and $0$s elsewhere (e.g. $I_{2,3} = \begin{bmatrix}1&0&0\\0&1&0\end{bmatrix}$). 
Following the proof of Theorem \ref{thm:converse2}, we see that the closure of the $G(\bd)$-orbit through \eqref{eq:specialsymmetricquiverrep} contains an open set which is a homogeneous fiber bundle over the symmetric matrix Schubert variety $\overline{Bw}\subseteq \text{Mat}_+^\circ(6,6)$. Hence the symmetric quiver orbit closure is neither normal nor Cohen-Macaulay.
\end{example}

In fact, the above example coincides with Pin's example from \S7.2.2 of his thesis (up to considering the same orbit in $GL(6)/K$ instead of $SL(7)/SO(7)$). In  \cite[\S7.2.2.]{Pin01} there is a generalization of such an example which produces a family of non normal and non Cohen-Macaulay Borel orbit closures in symmetric varieties of type $SL(2t+1)/SO(2t+1)$, so that applying the converse theorem we obtain a family of symmetric quiver orbit closures which are neither normal nor Cohen-Macaulay.

\section{Symmetric Zelevinsky permutations}\label{sec:zperms}
The primary objective of this section is to prove Theorem \ref{thm:combin}. We begin by providing an overview of its proof (Section \ref{sec:Zpermoverview}), delaying the details until later (Sections \ref{sect:Zperm} and \ref{sect:mainZperm2}).  We work entirely in the bipartite setting in this section, aside from the explanation in Remark \ref{rem:posetmapallorientations} for how the first part of Theorem \ref{thm:combin} follows for arbitrary orientation from the bipartite case.

Let $Q^\epsilon$ be a symmetric bipartite quiver as in \eqref{eq:typea}, where $\epsilon=s(a)$ with $a$ the unique arrow fixed by the involution of $Q$.  Fix a symmetric dimension vector $\bd$ and let $d = d_x = d_y$. Let $Y_d^\epsilon\subseteq \text{Mat}^\circ_\epsilon(2d,2d)$ denote the matrix space from \eqref{eq:MatrixMSSlice} as usual. 
Recall that $GL(2d)$ acts on $\text{Mat}^\circ_\epsilon(2d,2d)$, by $g\cdot M = gMg^t$, thus so does the Borel subgroup $B_-$. These $B_-$-orbits are indexed by involutions in $\mathfrak{S}_{2d}$ (when $\epsilon = 1$) and fixed point free involutions in $\mathfrak{S}_{2d}$ (when $\epsilon = -1$) as explained in Section \ref{sect:combBorel}. 
We identify elements of $\mathfrak{S}_{2d}$ with their permutation matrices, following the conventions in Section \ref{subsec:permConventions}. Given a permutation matrix $u$, where $u$ is an involution, we use the notation $[u]_\epsilon$ to denote the matrix obtained from $u$ by multiplying all $1$s below the main diagonal by $\epsilon$ as in the proof of Proposition \ref{prop:POrbitNWRanks}. Observe that if $u$ is the matrix of a fixed point free involution in $\mathfrak{S}_{2d}$ then $[u]_{\epsilon}\in \text{Mat}^\circ_{\epsilon}(2d,2d)$.

Throughout this section, we partition each $M\in \Mat_\epsilon^\circ(2d,2d)$ into blocks, where the block rows and block columns have sizes (from left to right and top to bottom, respectively) 
\begin{equation}
\bd(y_0), \bd(y_1),\dots, \bd(y_{n-1}), \bd(x_n),\dots, \bd(x_1).
\end{equation}

\subsection{Overview}\label{sec:Zpermoverview}
We now outline the main steps in the proof of Theorem \ref{thm:combin}, carried out in the following sections. Section \ref{sect:Zperm} begins by reviewing the ordinary (bipartite) Zelevinsky permutation from \cite{KR15}, motivating Definition \ref{def:Zperm} of the symmetric (bipartite) Zelevinsky permutation.
The main result of this section is a proof of the following theorem, which geometrically characterizes the symmetric Zelevinsky permutation of a representation via the symmetric Zelevinsky map.

Given a symmetric orbit closure $\mathscr{O}\subseteq \srep_Q(\bd)$, let $\mathscr{O}^\dagger\subseteq G/K$ be as in the statement of the main theorem. Let $\iota_\epsilon :G/K\rightarrow \text{Mat}^\circ_\epsilon(2d,2d)$ be the isomorphism described in equation \eqref{eqn:MPembedding},  and, as in the proof of the main theorem, let $\mathscr{O}^\ddagger=\iota_\epsilon(\mathscr{O}^\dagger)$.

\begin{theorem}\label{thm:mainZperm1} Let $\mathscr{O}\subseteq \srep_Q(\bd)$ be a symmetric orbit closure.
The symmetric Zelevinsky permutation $v^\epsilon(\mathscr{O})$ from Definition \ref{def:Zperm} is the unique element of $I^\epsilon \subset \mathfrak{S}_{2d}$ such that $\mathscr{O}^\ddagger = \overline{B_-\cdot [v^\epsilon(\mathscr{O})]_\epsilon}$.
\end{theorem}

From this theorem, we obtain a well-defined map $v^\epsilon(\cdot)$ from the set of symmetric orbit closures in $\srep_Q(\bd)$ to $\mathfrak{S}_{2d}$. 
It also follows that the map $v^\epsilon(\cdot)$ is   order-reversing and injective; this is explained in Corollary \ref{cor:ordRevInj}. This completes the proof of the first statement of Theorem \ref{thm:combin}.

In Section \ref{sect:mainZperm2}, we precisely identify the image of $v^\epsilon(\cdot)$ in $\mathfrak{S}_{2d}$, thus giving a combinatorial model of  symmetric type $A$ quiver orbit closures.
To describe the image, we need to introduce some additional notation related to the symmetric group.

Let $v_\square\in \mathfrak{S}_{2d}$ denote the fixed point free involution given in one-line notation by 
\begin{equation}
    v_\square = d+1,\dots,2d, 1,2,\dots d.
\end{equation}
Let $v^\epsilon_{\text{max}}$ be the symmetric Zelevinsky permutation that Theorem \ref{thm:mainZperm1} associates to the unique maximal $G(\mathbf{d})$-orbit in $\text{srep}_Q(\mathbf{d})$.
Let $[v^\epsilon_{\text{max}}, v_\square]$ denote the Bruhat interval of all permutations $u\in \mathfrak{S}_{2d}$ satisfying $v^\epsilon_{\text{max}}\leq u\leq v_{\square}$ in Bruhat order.

Let $\mathcal{W}_P := \mathfrak{S}_{\bd(y_0)}\times \mathfrak{S}_{\bd(y_1)}\times \cdots \times \mathfrak{S}_{\bd(y_{n-1})}\times \mathfrak{S}_{\bd(x_n)}\times \cdots \times \mathfrak{S}_{\bd(x_2)}\times\mathfrak{S}_{\bd(x_1)}$. This is naturally a subgroup of $\mathfrak{S}_{2d}$ consisting of permutation matrices down the block diagonal. Given any $u\in \mathfrak{S}_{2d}$, the double coset $\mathcal{W}_Pu\mathcal{W}_P$ has a unique element of minimal length. Let $\mathcal{W}^P$ denote the set of these minimal length double coset representatives. 

Recall that $I^+$ (respectively $I^-$) denote the set of involutions (respectively fixed point free involutions) in $\mathfrak{S}_{2d}$. 
Let $\mathcal{J}$ denote the set of  elements in $\mathcal{W}^P\cap I^+$ which have an even number of $1$s in each block centered along the main matrix diagonal. Let $u\in \mathcal{J}$. Since $u\in \mathcal{W}^P$,
    the $1$s in $u$  appear northwest to southeast down block rows and block columns. Thus, since $u$ is also an involution,  
    the (even number of) $1$s in each block along the diagonal of $u$ appear in locations $(i_p,i_p),(i_p+1,i_p+1), (i_p+2,i_p+2)\dots, (i_p+2r_p+1, i_p+2r_p+1)$ for some integers $i_p, r_p$. Let $\alpha(u)$ be the permutation obtained from $u\in \mathcal{J}$ by replacing these $1$s down the diagonal with $1$s in positions $(i_p,i_p+1),(i_p+1,i_p),(i_p+2,i_p+3),(i_p+3,i_p+2),\dots,(i_p+2r_p,i_p+2r_p+1),(i_p+2r_p+1,i_p+2r_p)$. Note that  $\alpha(u)$ is a fixed point free involution. Furthermore, it is clear by construction that $\alpha:\mathcal{J}\rightarrow I^-$ is a well-defined injective map.

    \begin{example}
In the involution $u$ below, the blocks centered along the matrix diagonal are square of sizes $1,3,1,3$ (listed from northwest to southeast). There are an even number of $1$s in each of these diagonal blocks, and when $1$s appear in these blocks, they appear in consecutive positions down the diagonal (i.e., in positions $(7,7), (8,8)$). The operation $\alpha$ replaces the $1$s in locations $(7,7), (8,8)$ with $1$s in locations $(7,8), (8,7)$.
\begin{equation}\label{eq:permeg}
  u = 
 \left[
 \begin{array}{c|ccc|c|ccc}
      0&1&0&0&0&0&0&0\\
      \hline
      1&0&0&0&0&0&0&0\\
      0&0&0&0&1&0&0&0\\
      0&0&0&0&0&1&0&0\\
      \hline
      0&0&1&0&0&0&0&0\\
      \hline
      0&0&0&1&0&0&0&0\\
      0&0&0&0&0&0&1&0\\
      0&0&0&0&0&0&0&1
 \end{array}
 \right], \quad
\alpha(u) = 
 \left[
 \begin{array}{c|ccc|c|ccc}
      0&1&0&0&0&0&0&0\\
      \hline
      1&0&0&0&0&0&0&0\\
      0&0&0&0&1&0&0&0\\
      0&0&0&0&0&1&0&0\\
      \hline
      0&0&1&0&0&0&0&0\\
      \hline
      0&0&0&1&0&0&0&0\\
      0&0&0&0&0&0&0&1\\
      0&0&0&0&0&0&1&0
 \end{array}
 \right].
 \end{equation}
\end{example}

We can now define sets $\mathcal{A}_\epsilon$ which will be the images of the maps $v^\epsilon(\cdot)$. 
\begin{enumerate}[(i)]
    \item Let $\mathcal{A}_{+} := [v^\epsilon_{\text{max}},v_\square]\cap \mathcal{W}^P\cap I^{+}$.
    \item Let $\mathcal{A}_{-}:= \{\alpha(u)\mid u\in [v^\epsilon_{\text{max}},v_\square]\cap \mathcal{J}\}$.
\end{enumerate}
Both sets $\mathcal{A}_+$ and $\mathcal{A}_{-}$ are partially ordered by Bruhat order.

\begin{theorem}\label{thm:mainZperm2}
The map $v^\epsilon(\cdot)$ is order-reversing with image exactly the set of permutations $\mathcal{A}_\epsilon\subseteq \mathfrak{S}_{2d}$ above. Specifically, symmetric orbit closures $\mathscr{O}_1, \mathscr{O}_2\subseteq \srep_Q(\bd)$ satisfy $\mathscr{O}_1\subseteq \mathscr{O}_2$ if and only if $v^\epsilon(\mathscr{O}_1)\geq v^\epsilon(\mathscr{O}_2)$.
\end{theorem}

Assuming the results above, we have proven Theorem \ref{thm:combin} in the bipartite setting, and then the following remark completes the proof in general.

\begin{remark}\label{rem:posetmapallorientations}
For an arbitrary symmetric type $A$ quiver, there is an associated bipartite symmetric type $A$ quiver.  By the results in Section \ref{sec:orientation}, 
each poset of orbit closures for the original symmetric quiver will be a subset (in fact, an upwards closed order ideal) of a poset of orbit closures for the associated bipartite quiver.  The injective, order-preserving map for the bigger poset then simply restricts to an injective, order-preserving map on the small poset.
\end{remark}

\begin{remark}
    We expect that having explicitly-described symmetric Zelevinsky permutations will be useful for multiple reasons including:
    \begin{enumerate}[(i)]
        \item Producing explicit combinatorial formulas for equivariant cohomology or $K$-theory of orbit closures in $\text{srep}_Q(\bd)$: Wyser and Yong \cite{WyserYong17} and Marberg and Pawlowski \cite{MP19} have found formulas for cohomology or $K$-theory representatives of Borel orbit closures in $G/K$. We expect that these formulas, together with the material in the present paper, can be combined to produce combinatorial formulas for orbit closures in symmetric type $A$ representation spaces. A number of analogous results exist in the usual (i.e., non-symmetric) type $A$ setting (see \cite{BFchernclass, MR1932326,FRdegenlocithom,BFR,KKR, KMS,yongComponent, MR2137947, MR2114821, MR2306279}).
        \item Understanding singularities via pattern avoidance criteria: by our main theorem, the closure of an orbit $\mathscr{O}$ in $\srep_Q(\bd)$ is smoothly equivalent to the Borel orbit closure in $G/K$ corresponding to the involution $v^\epsilon(\mathscr{O})$ (under the poset isomorphism $\varphi$ from Proposition \ref{prop:RS-BOrbitsInvolutions}); by  \cite[Theorem 4.9, Theorem 4.10]{McG22} there is a pattern avoidance criterion for the smoothness and rationally smoothness of such an orbit closure. Thus  it is possible to conclude that an orbit closure $\mathscr{O}$ is (rationally) smooth as soon as its symmetric Zelevinsky permutation $v^\epsilon(\mathscr{O}
        )$  avoids (in the symmetric sense, as explained in \cite[Section 4.1]{McG22}) the patterns indicated by McGovern.
    \end{enumerate}
\end{remark}

\subsection{Symmetric Zelevinsky permutations and the proofs of Theorems \ref{thm:mainZperm1} and \ref{thm:mainZperm2}}\label{sect:Zperm}
The symmetric Zelevinsky permutation of a symmetric type $A$ quiver representation is closely related to the Zelevinsky permutation of the associated ordinary type $A$ quiver representation, so we review this first.

Let $Q$ be a bipartite type $A$ quiver and $\bd$ a dimension vector for $Q$. 
As shown in \cite{KR15} (see also Section \ref{sec:ordinaryZmap}), each $GL({\bf d})$-orbit closure $\mathscr{O}\subseteq \text{rep}_Q(\bd)$ is isomorphic, up to  smooth factor, to an open subvariety of a type $A$ Schubert variety. The permutation associated to this Schubert variety is called the \emph{bipartite Zelevinsky permutation} in \cite{KR15}. We recall its construction in Proposition \ref{prop:zperm} below. 


Let $d_x = \sum_{i = 1}^n \bd(x_i)$ and $d_y = \sum_{i=0}^{n-1}\bd(y_i)$.  We partition each $M\in \Mat(d_x+d_y,d_x+d_y)$ into blocks as follows. Block rows and columns are labeled by vertices of $Q$ (recall the notation in \eqref{eq:typea}) and the sizes of these blocks are  determined by the dimension vector $\bd$. 
Precisely, we label block rows from top to bottom by $y_0,y_1,\dots, y_{n-1},$ $ x_n,\dots, x_2,x_1$. We label block columns from left to right by $x_n,\dots, x_2, x_1, y_0, y_1,\dots, y_{n-1}$. A block row or column labeled by vertex $z$ has size $\bd(z)$.
Observe that there are $2n$ block rows and $2n$ block columns. Let  $M_{[p],[q]}$ be the northwest justified submatrix of $M$ consisting of the intersection of block rows $[p]:=\{1,2, \ldots, p\}$ with block columns $[q]$ and let ${\bar r}_{p,q}(M):=\rank M_{[p],[q]}$. 
Let $\zeta: \rep_Q(\bd)\rightarrow Y_{\bd}$ be the bipartite Zelevinsky map (see \eqref{eq:bipartiteZmap}).

\begin{proposition}\label{prop:zperm} \cite[Prop. 4.8]{KR15}
Let $W\in \rep_Q(\bd).$
There exists a unique $(d_x+d_y)\times (d_x+d_y)$ permutation matrix $v(W)$, called the bipartite Zelevinsky permutation matrix, that satisfies the following conditions:
\begin{enumerate}[(i)]
\item  the number of 1s in block $(p,q)$ 
of $v(W)$ is 
\begin{equation}\label{eq:nwbr}
{\bar r}_{p, q}(\zeta(W)) + {\bar r }_{p-1, q-1}(\zeta(W)) - {\bar r}_{p-1, q}(\zeta(W))- {\bar r}_{p, q-1}(\zeta(W))
\end{equation}
where $\zeta$ denotes the bipartite Zelevinsky map and a summand is taken to be 0 if either index in the subscript is outside of the range $\{1, \dotsc, 2n\}$;
\item the 1s in $v(W)$ are arranged from northwest to southeast across each block row;
\item the 1s in $v(W)$ are arranged from northwest to southeast down each block column.
\end{enumerate}
Furthermore, $W'\in \rep_Q(\bd)$ is in the same $\text{GL}({\bf d})$-orbit as $W$ if and only if $v(W') = v(W)$.
\end{proposition}

In Proposition \ref{prop:zpermProperties} below, item (i) is immediate from the construction of the bipartite Zelevinsky permutation and item (ii) is \cite[Lem.  4.10]{KR15}.

\begin{proposition}\label{prop:zpermProperties}
For $W\in \rep_Q(\bd)$, we have that $v(W)$ satisfies the following properties:
\begin{enumerate}[(i)]
\item $
{\bar r}_{p,q} (\zeta(W)) = {\bar r}_{p,q} (v(W))
$ for all $1\leq p,q\leq 2n$;
\item Each element of $\mathcal{E}(v(W))$ occurs in the southeast corner of a block.
\end{enumerate}
\end{proposition}

Now let $Q^\epsilon$ be a symmetric bipartite quiver as in  \eqref{eq:typea} with symmetric dimension vector $\bd$.
Using the embedding \eqref{eq:sreptorep}, we get the bipartite Zelevinsky permutation $v(W)$ for any $W\in \srep_Q(\bd)$.
Recall the symmetric Zelevinsky map $\zeta^\epsilon: \srep(\bd)\rightarrow Y^\epsilon_d$ defined in \eqref{eq:zetaKdef} and that $\zeta^{\epsilon}(W)$ is obtained from $\zeta(W)$ by multiplying certain rows and columns by $\epsilon$.

\begin{proposition}\label{prop:symZperm123}
Let $W\in \srep_Q(\bd)$. Then $v(W)\in \mathfrak{S}_{2d}$ is an involution and $W'\in \srep_Q(\bd)$ is in the same $G(\bd)$-orbit as $W$ if and only if $v(W') = v(W)$.
\end{proposition}

\begin{proof}
To prove that $v(W)$ is an involution, we will show that it is symmetric, which is enough since a symmetric permutation matrix is its own inverse.  Since we have
\begin{equation}\label{eq:NWblockrankEquality}
    {\bar r}_{p,q}(\zeta(W)) = {\bar r}_{p,q} (\zeta^\epsilon(W))= {\bar r}_{p,q}(\zeta^\epsilon(W)^T) = {\bar r}_{q,p} (\epsilon\zeta^\epsilon(W)) = {\bar r}_{q,p}(\zeta(W))
\end{equation}
for all $1\leq p,q\leq 2n$,
Proposition \ref{prop:zperm}(i) implies that the number of $1$s in block $(p,q)$ of $v(W)$ is equal to the number of $1$s in block $(q,p)$. 

Now within block $(p,q)$, the positions of these $1$s are determined by 
Proposition \ref{prop:zperm}(ii) and (iii).
One way to see how these conditions determine the placement of 1s in the block is to proceed from the northwest and proceeding to the southeast within the block, 
placing a 1 in position $(i,j)$ exactly when there is not already a 1 in row $i$ with lower column index, not already a 1 in column $j$ with lower row index, and the total number of $1$s in the block has not already been achieved.
For example, in block $(1,1)$ this forces the 1s to be consecutively along the diagonal starting in position $(1,1)$ (if there are any 1s).

Since this description symmetric with respect to switching rows and columns, by induction there is a 1 in position $(k,l)\neq (i,j)$ with $k \leq i$ and $l \leq j$ exactly when there is a 1 in position $(l,k)$, so a 1 will be placed in position $(i,j)$ exactly when a 1 is placed in position $(j,i)$, making $v(W)$  symmetric matrix.


The second statement now follows from the ``furthermore'' part of Proposition \ref{prop:zperm} together with Theorem \ref{thm:DWorbits}.
\end{proof}

\begin{example}\label{ex:rankCheck}
We continue Example \ref{ex:runExStart} to illustrate the equality \eqref{eq:NWblockrankEquality} in the above proof. Consider the following symmetric quiver representation, where $\epsilon = -1$:
\begin{equation}
W =\vcenter{\hbox{ \begin{tikzpicture}[point/.style={shape=circle,fill=black,scale=.5pt,outer sep=3pt},>=latex]
   \node[outer sep=-2pt] (y0) at (-1,0) {$k$};
   \node[outer sep=-2pt] (x1) at (0,1) {$k^3$};
  \node[outer sep=-2pt] (y1) at (1,0) {$k^3$};
   \node[outer sep=-2pt] (x2) at (2,1) {$k$};
 
  \path[->]
  	(x1) edge node[left] {$A$} (y0) 
	(x1) edge node[left] {$B$} (y1)
  	(x2) edge node[right] {$A^t$} (y1);
   \end{tikzpicture}}},
   \quad
   (A,B) = \left(\begin{bmatrix}1&0&0\end{bmatrix}, \begin{bmatrix}0&1&0\\-1&0&0\\0&0&0\end{bmatrix}\right).
   \end{equation}
Then \[\zeta(W)  = \begin{bmatrix}0&A&I_{1\times 1}&0\\ A^t&B&0&I_{3\times 3}\\ I_{1\times 1}&0&0&0\\0&I_{3\times 3}&0&0\end{bmatrix}, \quad \zeta^\epsilon(W)  = \begin{bmatrix}0&A&I_{1\times 1}&0\\ -A^t&B&0&I_{3\times 3}\\ -I_{1\times 1}&0&0&0\\0&-I_{3\times 3}&0&0\end{bmatrix}.\]
Noting that $B^t = -B$, it is straightforward to verify that each equality in \eqref{eq:NWblockrankEquality} holds for all $1\leq p,q\leq 4$ by multiplying appropriate block rows and columns by $-1$.
\end{example}

Recall from Section \ref{sec:Zpermoverview} that $\mathcal{J}$ denotes the set of  elements in $\mathcal{W}^P\cap I^+$ which have an even number of $1$s in each block centered along the main matrix diagonal.

\begin{lemma}\label{lem:onesPositions}
Let $W\in \srep_Q({\bf d})$ and suppose $\epsilon = -1$.
Then the bipartite Zelevinsky permutation $v(W)\in \mathfrak{S}_{2d}$ is an element of $\mathcal{J}$.
\end{lemma}

\begin{proof}
That $v(W)\in \mathcal{W}^P$ follows by Proposition $\ref{prop:zperm}$ (ii) and (iii). That $v(W)\in I^+$ follows by Proposition \ref{prop:symZperm123}. So, it remains to show that for each $1\leq p\leq 2n$, the $(p,p)$ block of  $v(W)$  has an even number of $1$s.

Since ${\bar r}_{p,p}(\zeta(W))  = {\bar r}_{p,p}(\zeta^\epsilon(W))$ and since $\zeta^\epsilon(W)_{[p],[p]}$ is a skew-symmetric matrix, we conclude that ${\bar r}_{p,p}(\zeta(W))$ is even. 
By Proposition \ref{prop:zperm} (i), the number of $1$s in block $(p,p)$ of $v(W)$ is
\[
{\bar r}_{p,p}(\zeta(W)) + {\bar r}_{p-1,p-1}(\zeta(W)) - {\bar r}_{p-1,p}(\zeta(W)) - {\bar r}_{p,p-1}(\zeta(W)).\]
Since ${\bar r}_{p,p-1}(\zeta(W)) = {\bar r}_{p-1,p}(\zeta(W))$ from \eqref{eq:NWblockrankEquality}, and ${\bar r}_{p,p}(\zeta(W))$ and ${\bar r}_{p-1,p-1}(\zeta(W))$ are even numbers, it follows that there are an even number of $1$s in block $(p,p)$ of $v(W)$. 
 \end{proof}

We can now define the symmetric Zelevinsky permutation. Recall from Section \ref{sec:Zpermoverview} the map $\alpha:\mathcal{J}\rightarrow I^-$. 

\begin{definition}\label{def:Zperm}
Let $W\in \srep_Q(\bd)$. The \textbf{symmetric Zelevinsky permutation} $v^\epsilon(W)\in \mathfrak{S}_{2d}$ is defined as follows: 
\begin{itemize}
    \item If $\epsilon = 1$, define $v^\epsilon(W) := v(W)$. This is an involution by Proposition \ref{prop:symZperm123}.
    \item If $\epsilon = -1$, then $v(W)\in \mathcal{J}$ by Lemma \ref{lem:onesPositions}. Define $v^\epsilon(W):= \alpha(v(W))$, which is a fixed point free involution.
\end{itemize}
\end{definition}

\begin{example}
We continue Example \ref{ex:rankCheck} and construct $v^\epsilon(W)$.
We first construct the matrix of northwest block ranks ${\bf {\bar r}} = \left({\bar r}_{p,q}(\zeta(W))\right)_{1\leq p,q\leq 4}$. Then we apply the three conditions in Proposition \ref{prop:zperm} to obtain the bipartite Zelevinsky permutation $v(W)$:
\[
 {\bf {\bar r}}= \begin{bmatrix}0&1&1&1\\1&2&3&4\\1&3&4&5\\ 1&4&5&8\end{bmatrix}, \quad
  v(W) = 
 \left[
 \begin{array}{c|ccc|c|ccc}
      0&1&0&0&0&0&0&0\\
      \hline
      1&0&0&0&0&0&0&0\\
      0&0&0&0&1&0&0&0\\
      0&0&0&0&0&1&0&0\\
      \hline
      0&0&1&0&0&0&0&0\\
      \hline
      0&0&0&1&0&0&0&0\\
      0&0&0&0&0&0&1&0\\
      0&0&0&0&0&0&0&1
 \end{array}
 \right].
\]
Observe that $v(W)$ is the permutation that appears on the left in \eqref{eq:permeg}. Since $\epsilon = -1$, we must apply $\alpha$ to get $v^\epsilon(W)$, which is the permutation on the right in \eqref{eq:permeg}.
\end{example}

Now that we have defined the symmetric Zelevinsky permutation, our next goal is to prove Theorem \ref{thm:mainZperm1}. We begin with some preliminary results.

We note  that in the next proposition we need to consider arbitrary northwest ranks  rather than just block ranks. Recall that for a matrix $M\in\Mat(2d,2d)$ we denote by $r_{i,j}(M)$ the rank of the submatrix of $M$ consisting of rows $\{1, \ldots, i\}$ intersected with columns $\{1, \ldots, j\}$.

\begin{lemma}\label{lem:skewSymRankComparison}
    Let $\epsilon = -1$, $W\in \srep_Q(\bd)$, and let $M\in \text{Mat}^\circ_{\epsilon}(2d,2d)$. Then $r_{i,j}(M)\leq r_{i,j} (v(W))$ for all $1\leq i,j\leq 2d$ if and only if $r_{i,j}(M)\leq r_{i,j}(v^\epsilon(W))$ for all $1\leq i,j\leq 2d$. 
\end{lemma}

\begin{proof}
If $v^\epsilon(W) = v(W)$ then there is nothing to show. So assume otherwise.

Let $s_i\in \mathfrak{S}_{2d}$ denote the simple transposition swapping $i$ and $i+1$. By Lemma \ref{lem:onesPositions} and the discussion following the definition of $\mathcal{J}$ in Section \ref{sec:Zpermoverview}, we may assume that there are $j_1, j_2, \ldots, j_m$ such that the locations of the $1$s in the (non-symmetric) bipartite Zelevinsky permutation $v(W)$ are $(j_1,j_1), (j_1+1,j_1+1)$, $(j_2,j_2), (j_2+1,j_2+1)$, $\dots$, $(j_m,j_m),(j_m+1,j_m+1)$. We assume that $j_k+1<j_{k+1},$ for each $1\leq k\leq m-1$, so that $v(W)$ has $2m$-many $1$s in positions along the diagonal.
Then, by the construction of $v^\epsilon(W)$ from $v(W)$, we see that 
\[v^\epsilon(W) =  s_{j_m}\cdots s_{j_1}v(W).\]

Observe that $r_{i,j}(v(W))$ and $r_{i,j}(s_{j_1}v(W))$ are equal for all $1\leq i,j\leq 2d$ except $(i,j) = (j_1,j_1)$. Precisely, $r_{j_1,j_1}(s_{j_1}v(W)) = r_{j_1,j_1}(v(W))-1$. 

Now, if $M$ is a skew-symmetric matrix, then so is the northwest $j_1 \times j_1$ submatrix of $M$, and  hence that submatrix has even rank. 
Additionally, since the northwest-most $1$ that appears along the diagonal in $v(W)$ appears in location $(j_1,j_1)$, and $v(W)$ is an involution, we observe that $r_{j_1,j_1}(v(W))$ is odd. 
Consequently, $r_{j_1,j_1}(M)\leq r_{j_1,j_1}(v(W))$ if and only if $r_{j_1,j_1}(M)\leq r_{j_1,j_1}(s_{j_1}v(W))$. Thus, $r_{i,j}(M)\leq r_{i,j}(v(W))$ for all $1\leq i,j\leq 2d$ if and only if $r_{i,j}(M)\leq r_{i,j}(s_{j_1}v(W))$ for all $1\leq i,j\leq 2d$.

A similar argument shows that, for each $1\leq k\leq m-1$, we have $r_{i,j}(M)\leq r_{i,j}(s_{j_k}\cdots s_{j_1}v(W))$ for all $1\leq i,j\leq 2d$ if and only if $r_{i,j}(M)\leq r_{i,j}(s_{j_{k+1}}\cdots s_{j_1}v(W))$ for all $1\leq i,j\leq 2d$. 

The desired result now follows.
\end{proof}

 Let $P\leq GL(2d)$ denote the parabolic subgroup of block lower triangular matrices with block sizes $\bd(y_0),$ $ \bd(y_1), \dots,$ $ \bd(y_{n-1}), $ $\bd(x_n),\dots,$ $ \bd(x_2),\bd(x_1)$ from northwest to southeast down the diagonal. Let $\text{Mat}^\circ_{\epsilon}(2d,2d)$ denote the group of invertible symmetric matrices of size $2d\times 2d$ when $\epsilon = 1$, and the group of invertible skew-symmetric matrices of size $2d\times 2d$ when $\epsilon = -1$. We have the following symmetric quiver analogue of Proposition \ref{prop:zpermProperties}.

\begin{proposition}\label{prop:symmZpermEssentialBox}
Let $W\in \srep_Q(\bd)$. The symmetric bipartite Zelevinsky permutation $v^\epsilon(W)$ satisfies the following properties:
\begin{enumerate}[(i)]
\item ${\bar r}_{p,q}(\zeta^\epsilon(W)) = {\bar r}_{p,q}(v^\epsilon(W))$, for all $1\leq p,q\leq 2n.$
\item $\overline{B_-\cdot [v^{\epsilon}(W)]_{\epsilon}} = \overline{P\cdot [v^\epsilon(W)]_\epsilon}$, where the orbit closure is taken in $\text{Mat}^\circ_{\epsilon}(2d,2d)$. 
\end{enumerate}
\end{proposition}

\begin{proof}
To prove (i), note that
\[{\bar r}_{p,q}(\zeta^\epsilon(W))= {\bar r}_{p,q}(\zeta(W)), \quad \text{ and } \quad {\bar r}_{p,q}(v^\epsilon(W)) = {\bar r}_{p,q}(v(W)),\] for each $1\leq p,q\leq 2n$. Thus, (i) follows from Proposition \ref{prop:zpermProperties} (i).

We next prove (ii). Since the base field $\kk$ is algebraically closed and both orbit closures are algebraic varieties, it suffices to check that the sets of $\kk$-points agree. First, since $B_-\leq P$, we have $\overline{B_-\cdot [v^{\epsilon}(W)]_{\epsilon}} \subseteq \overline{P\cdot [v^\epsilon(W)]_\epsilon}$. So 
suppose that $M \in \overline{P\cdot [v^\epsilon(W)]_\epsilon}$. We wish to show that $M\in \overline{B_-\cdot [v^{\epsilon}(W)]_\epsilon}$. By Proposition \ref{prop:POrbitNWRanks}, we have ${\bar r}_{p,q}(M)\leq {\bar r}_{p,q}(v^\epsilon(W))$ for each $1\leq p, q \leq 2n$. Noting that $v(W) = v^\epsilon(W)$ (when $\epsilon = 1$) or applying Lemma \ref{lem:skewSymRankComparison} (when $\epsilon = -1$), we have that ${\bar r}_{p,q}(M)\leq {\bar r}_{p,q}(v(W))$ for each $1\leq p, q \leq 2n$. 

By Proposition \ref{prop:zpermProperties} (ii), each element of the essential set, $\mathcal{E}(v(W))$, is supported in the southeast corner of a block. Hence $r_{i,j}(M)\leq r_{i,j}(v(W))$, for $(i,j)\in \mathcal{E}(v(W))$, and hence $r_{i,j}(M)\leq r_{i,j}(v(W))$, for all $1\leq i,j\leq 2d$ by \cite[Lemma 3.10]{Ful92}. 
Noting once more that $v(W) = v^\epsilon(W)$ (when $\epsilon = 1$) or applying Lemma \ref{lem:skewSymRankComparison} (when $\epsilon = -1$), we have $r_{i,j}(M)\leq r_{i,j}(v^\epsilon(W))$ for all $1\leq i,j\leq 2d$.
Hence $M\in \overline{B_-\cdot [v^{\epsilon}(W)]_\epsilon}$ by Proposition \ref{prop:matrixNWranks}.
\end{proof}

We are now ready to prove Theorem \ref{thm:mainZperm1}.

\begin{proof}[Proof of Theorem \ref{thm:mainZperm1}]
Let $W\in \srep_Q(\bd)$. 
Let $\mathscr{O} := \overline{G(\bd)\cdot W}\subseteq \srep_Q(\bd)$.  
First recall that $\zeta^\epsilon$ is a closed immersion which restricts to an isomorphism from $\mathscr{O}$ to $\zeta^\epsilon(\mathscr{O})$. Thus, we have that $\overline{\zeta^\epsilon(G(\bd)\cdot W)} = \zeta^\epsilon(\overline{G(\bd)\cdot W}) = \zeta^\epsilon(\mathscr{O})$. 
Next recall from the proof of Theorem \ref{thm:mainTheorem} 
that $\mathscr{O}^{\ddagger}:=\overline{P\cdot \zeta^\epsilon(\mathscr{O})}\subset\Mat^\circ_\epsilon(2d,2d)$. Then, 
\[
\mathscr{O}^\ddagger = \overline{P\cdot\zeta^\epsilon(\mathscr{O})}=
\overline{P\cdot\zeta^\epsilon(G(\bd)\cdot W)}=
\overline{P\cdot \zeta^\epsilon(W)}.
\]
where the second equality follows from the equivariance of $\zeta^\epsilon$ with respect to the embedding ${\bar \xi}$ in \eqref{eq:GdtoGL} of $G(\bd)$ into $P$.

Thus we have
\begin{equation}\label{eq:ZpermThm1eq2}
\mathscr{O}^\ddagger= \overline{P\cdot \zeta^\epsilon(W)} = \overline{P\cdot [v^\epsilon(W)]_\epsilon} = \overline{B\cdot [v^\epsilon(W)]_\epsilon},\end{equation}
where the middle equality follows from Proposition \ref{prop:symmZpermEssentialBox} (i) together with Proposition \ref{prop:POrbitNWRanks}, and the rightmost equality follows from Proposition \ref{prop:symmZpermEssentialBox} (ii).

Uniqueness of $v^\epsilon(W)$ follows from the $B_-$-orbit closure parametrization in $\Mat_\epsilon^\circ(2d,2d)$ given in Proposition \ref{prop:BOrbitsMatrixSpaceInvolutions}: as recalled also at the beginning of this section, if $O\subset \Mat_\epsilon^\circ(2d,2d)$ is a $B_-$-orbit closure, then there exists a unique $u\in I^\epsilon$ such that $O=\overline{B\cdot [u]_\epsilon}$. \end{proof}

\begin{corollary}\label{cor:ordRevInj}
    $v^\epsilon(\cdot)$ is an order-reversing and injective map from the set of orbit closures in $\srep_Q(\bd)$, partially ordered by inclusion, to $\mathfrak{S}_{2d}$, partially ordered by Bruhat order. 
\end{corollary}

\begin{proof}
    By Theorem \ref{thm:mainZperm1}, $v^\epsilon(\cdot)$ is injective. 
    To see that it is order reversing, let $W_1, W_2\in \text{srep}_Q(\bd)$. 
    Let $\mathscr{O}_1 = \overline{G(\bd)\cdot W_1}$ and let $\mathscr{O}_2 = \overline{G(\bd)\cdot W_2}$.
    Since
    $\mathscr{O}^\ddagger= \overline{B\cdot [v^\epsilon(W_i)]_\epsilon}$, $i\in\{1,2\}$, we see that $\mathscr{O}_1\subseteq \mathscr{O}_2$ in $\text{srep}_Q(\bd)$ if and only if $\overline{B\cdot [v^\epsilon(W_1)]_\epsilon}\subseteq \overline{B\cdot [v^\epsilon(W_2)]_\epsilon}$, which is true if and only if $v^\epsilon(W_1)\geq v^\epsilon(W_2)$ in Bruhat order.
\end{proof}

\subsection{Image of $v^\epsilon(\cdot)$}\label{sect:mainZperm2}

Throughout this section we let $Q$ be a bipartite type $A$ quiver and let $\bd$ be a symmetric dimension for $Q$. 
As in Section \ref{sec:Zpermoverview}, let $\mathcal{W}_P := \mathfrak{S}_{\bd(y_0)}\times \mathfrak{S}_{\bd(y_1)}\times \cdots \times \mathfrak{S}_{\bd(y_{n-1})}\times \mathfrak{S}_{\bd(x_n)}\times \cdots \times \mathfrak{S}_{\bd(x_2)}\times\mathfrak{S}_{\bd(x_1)}\leq \mathfrak{S}_{2d}$.
and let $\mathcal{W}^P$ denote the set of minimal length representatives of double cosets in $\mathcal{W}_P\backslash\mathfrak{S}_{2d}/\mathcal{W}_P$.

\begin{proposition}\label{prop:bruhat}\cite[Thm. 4.12 (3)]{KR15}
Let $Z$ be an element of the open orbit in $\rep_Q(\bd)$. Let $0\in \rep_Q(\bd)$ denote the representation with all matrix entries being zeros. 
    There is an inclusion reversing bijection between the poset of orbit closures in $\rep_Q(\bd)$, partially ordered by inclusion, and  $[v(Z), v(0)]\cap \mathcal{W}^P$ partially ordered by Bruhat order. This bijection is given by:
    \begin{equation}\label{eq:ZmaptoPerm}
    \overline{GL(\bd)\cdot W}\mapsto v(W).
    \end{equation}
\end{proposition}

We will need one more lemma to prove Theorem \ref{thm:mainZperm2}. Recall from Section \ref{sec:Zpermoverview} that $\mathcal{A}_{+} = [v^\epsilon_{\text{max}},v_\square]\cap \mathcal{W}^P\cap I^{+}$.

\begin{lemma}\label{lem:lemmaForZperm2}
Let $u\in \mathcal{A}_+$. 
\begin{enumerate}[(i)]
    \item Then there is some $W\in \srep_Q(\bd)$, with $\epsilon = +1$, such that $u = v(W)$.
    \item If additionally $u\in \mathcal{A}_+$ has an even number of $1$s in each block diagonal, then there is some $W\in \srep_Q(\bd)$, with $\epsilon = -1$ such that $u = v(W)$.
\end{enumerate}
\end{lemma}

\begin{proof}
Since $u\in \mathcal{A}_+$, Proposition \ref{prop:bruhat} implies that $u = v(V)$ for some (ordinary) bipartite type $A$ representation $V\in \rep_Q(\bd)$.  

Since $u = v(V)$ is an involution, \cite[Lemma A.3]{KR15} implies that
$r_{J}(V) = r_{\tau J}(V)$, 
for all intervals $J\subseteq Q$. Item (i) now follows from Proposition \ref{prop:regtosym}.

For (ii), assume that $u = v(V)$ has an even number of $1s$ down each block diagonal. Then, since $u$ is an involution, we conclude that each northwest block rank ${\bar r}_{p,p}(v)$ for each $1\leq p\leq 2n$ is even. Finally, using 
\cite[Lemma A.3]{KR15}, we see that $r_{J}(V)$ is even for each interval $J$ such that $J = \tau J$. Applying item (ii) of Proposition \ref{prop:regtosym} completes the proof.
\end{proof}

We are now ready to prove Theorem \ref{thm:mainZperm2}.

\begin{proof}[Proof of Theorem \ref{thm:mainZperm2}]
The order reversing statement is Corollary \ref{cor:ordRevInj}, so it remains to describe the image.
We first prove the result for $\epsilon = +1$. To show that the image of $v^\epsilon$ is equal to the set $\mathcal{A}_{+} = [v^\epsilon_{\text{max}},v_\square]\cap \mathcal{W}^P\cap I^{+}$, we prove both containments. 
Let $u\in \mathcal{A}_{+}$, then $u = v(W)$ for some $W\in \srep_Q(\bd)$ by Lemma \ref{lem:lemmaForZperm2}. For the other inclusion, let $W\in \srep_Q(\bd)$. By Proposition \ref{prop:bruhat}, $v(W)\in [v^\epsilon_{\text{max}},v_\square]\cap \mathcal{W}^P$ and by Proposition \ref{prop:symZperm123} $v(W)$ is an involution. This completes the proof for $\epsilon=+1$ since $v(W)=v^\epsilon(W)$.

We now prove the result for $\epsilon = -1$. That is,  we will show that the image of $v^\epsilon$ is equal to the set $\mathcal{A}_- = \{\alpha(u)\mid u\in [v^\epsilon_{\text{max}},v_\square]\cap \mathcal{J}\}$ where the set $\mathcal{J}$ and the map $\alpha$ are as defined in Section \ref{sec:Zpermoverview}. 
First suppose that $w\in \mathcal{A}_-$. Then, $\alpha^{-1}(w)$ is an element of $\mathcal{A}_+$ and has an even number of $1$s in each block diagonal. Thus, but Lemma \ref{lem:lemmaForZperm2}, part (ii), there is a $W\in \srep_Q(\bd)$ (with $\epsilon = -1$) such that $\alpha^{-1}(w) = v(W)$. Hence $v^\epsilon(W) = \alpha(\alpha^{-1}(w)) = w$. Next suppose that $W\in \srep_Q(\bd)$. By Proposition \ref{prop:bruhat}, $v(W)\in [v^\epsilon_{\text{max}},v_\square]\cap \mathcal{W}^P$ and by Lemma \ref{lem:onesPositions}, $v(W)\in \mathcal{J}$. Thus, $\alpha(v(W))\in \mathcal{A}_-$. Since $v^\epsilon(W) = \alpha(v(W))$, we have $v^\epsilon(W)\in \mathcal{A}_-$.
\end{proof}

\bibliographystyle{alpha} 
\bibliography{symmetricquivers}

\end{document}